\documentclass[11pt]{amsart}
\usepackage{amsmath,amscd,amssymb,enumerate,verbatim,palatino,parskip}
\usepackage[T1]{fontenc}
\usepackage[backref]{hyperref}
\usepackage{graphicx}
\usepackage{tikz}
\usetikzlibrary{matrix}

\newcommand{\mM}{\mathcal{M}}

\newcommand{\CC}{\mathbf{C}}

\newcommand{\RR}{\mathbf{R}}
\newcommand{\QQ}{\mathbf{Q}}

\newcommand{\ZZ}{\mathbf{Z}}

\newcommand{\ev}{\mathrm{ev}}

\newcommand{\OP}{\operatorname}

\numberwithin{equation}{section}

\newtheorem{thm}[equation]{Theorem}
\newtheorem{lma}[equation]{Lemma}
\newtheorem{prp}[equation]{Proposition}
\newtheorem{cor}[equation]{Corollary}
\newtheorem{mainthm}{Theorem}
\newtheorem{maincor}[mainthm]{Corollary}
\newtheorem{clm}{Claim}

\theoremstyle{definition}

\newtheorem{dfn}[equation]{Definition}

\theoremstyle{remark}

\newtheorem{rmk}[equation]{Remark}

\newtheoremstyle{TheoremNum}
    {\topsep}{\topsep}              %%% space between body and thm
    {\itshape}                      %%% Thm body font
    {}                              %%% Indent amount (empty = no indent)
    {\bfseries}                     %%% Thm head font
    {.}                             %%% Punctuation after thm head
    { }                             %%% Space after thm head
    {\thmname{#1}\thmnote{ \bfseries #3}}%%% Thm head spec
\theoremstyle{TheoremNum}

\begingroup 
\makeatletter 
\@for\theoremstyle:=definition,remark,plain,TheoremNum\do{% 
\expandafter\g@addto@macro\csname th@\theoremstyle\endcsname{% 
\addtolength\thm@preskip\parskip 
}% 
} 
\endgroup 

\title[Unlinking of monotone Lagrangians]{Unlinking and unknottedness of monotone Lagrangian submanifolds}

\author[G. Dimitroglou Rizell]{Georgios Dimitroglou Rizell}
\address{Universit\'{e} Libre Bruxelles, CP218, Bd du Triomphe, B-1050 Bruxelles, Belgique}
\email{georgios.dimitroglou@ulb.ac.be}
\author[J. D. Evans]{Jonathan David Evans}
\address{Department of Mathematics, University College London, Gower Street, London WC1E 6BT, United Kingdom}
\email{j.d.evans@ucl.ac.uk}
\thanks{ Our collaboration was supported by a travel grant from the Contact And Symplectic Topology network (CAST) which is funded by the European Science Foundation (ESF). The first author was supported by the ERC Starting Grant of Fr{\'e}d{\'e}ric Bourgeois StG-239781-ContactMath.}

\begin{document}
\begin{abstract}
Under certain topological assumptions, we show that two mono\-tone Lagr\-ang\-ian submanifolds embedded in the standard symplectic vector space with the same monotonicity constant cannot link one another and that, individually, their smooth knot type is determined entirely by the homotopy theoretic data which classifies the underlying Lagrangian immersion. The topological assumptions are satisfied by a large class of manifolds which are realised as monotone Lagrangians, including tori. After some additional homotopy theoretic calculations, we deduce that all monotone Lagrangian tori in the symplectic vector space of odd complex dimension at least five are smoothly isotopic.
\end{abstract}
\maketitle
\part{Introduction}

Consider two $n$-dimensional embedded submanifolds, $L_1$ and $L_2$, of $2n$-di\-men\-si\-on\-al Euclidean space $\RR^{2n}$. We say that $L_1$ {\em links} $L_2$ if $L_1$ is homologically essential in the complement of $L_2$. We say that $L_1$ and $L_2$ are {\em not linked} if each one is nullhomologous in the complement of the other.

When $L_1$ and $L_2$ are diffeomorphic then we can ask if they are isotopic through embedded submanifolds. If they are not then we say they are {\em relatively knotted}. In high dimensions, $n\geq 4$, the issue of knottedness is related to the question of linking via a theorem of Haefliger and Hirsch \cite{HaeHir}. We discuss this in more depth in Section \ref{sect-HaeHir} below.

Now suppose that $\RR^{2n}$ is equipped with its standard symplectic form $\omega_0$ and that $L_1$ and $L_2$ are required to be Lagrangian submanifolds (that is, $\omega_0$ vanishes on the tangent spaces of $L_1$ and $L_2$). The question of whether $L_1$ and $L_2$ can be linked or relatively knotted is subtle and has been much studied, \cite{Bor,EliPol,EliPol2,EliPol3,EliPol4,Eva,Hin,HinIvr,Lut}. We review some known results in Section \ref{sect-context} below.

\section{Statement of results}

The theorems we prove concern knotting and linking of Lagrangian submanifolds in the standard symplectic vector space $\CC^n$, or more generally in a subcritical Stein manifold. We need to restrict attention to a certain class of Lagrangian embeddings: the monotone ones, see Definition \ref{dfn-monotone}. Note that by {\em torus} we always mean a product of circles.

\begin{mainthm}\label{mainthm-unlink}
Let $X$ be a subcritical Stein manifold and let $0<K_1\leq K_2$ be real numbers. An $K_1$-monotone Lagrangian torus embedded in the complement $C$ of an embedded $K_2$-monotone Lagrangian torus must be homologically trivial in $C$. In particular, two embedded $K$-monotone Lagrangian tori are each nullhomologous in the complement of the other.
\end{mainthm}

Using our definition of linking, the above theorem states that a $K_1$-monotone Lagrangian torus cannot link a $K_2$-monotone torus if $K_1\leq K_2$ and that monotone Lagrangian tori with the same monotonicity constant are not linked. This theorem fails if we let $K_2<K_1$: consider a pair of concentric circles with different radii in $\CC$. The theorem is clear for circles in $\CC$ by an area argument. It is less intuitively clear why it should be true in higher dimensions. Our proof fills the smaller Lagrangian by holomorphic discs of Maslov index two and uses a neck-stretching argument to ensure that these discs are disjoint from the bigger Lagrangian. The discs provide the necessary nullhomology of $L_1$ in the complement of $L_2$.

\begin{mainthm}\label{mainthm-unknot}
Suppose $L_1$ and $L_2$ are embedded monotone Lagrangian $n$-tori in $\CC^n$, $n\geq 4$, and suppose that they are homotopic through Lagrangian immersions. Then they are smoothly isotopic through embeddings (not necessarily Lagrangian embeddings).
\end{mainthm}

This is a consequence of Theorem \ref{mainthm-unlink} using Haefliger-Hirsch theory; this theory reduces questions about knottedness of middle-dimensional submanifolds of Euclidean space to computations of self-linking numbers. Roughly speaking, it associates to an embedding a certain homotopy class of normal vector fields and two embeddings are isotopic if and only if their Haefliger-Hirsch fields are homotopic (this is explained fully in Section \ref{sect-HaeHir}).

The key point in our proof is Lemma \ref{lma-HaeHir} which identifies the Haefliger-Hirsch field for a monotone Lagrangian torus purely in terms of the Maslov class. This is achieved by showing that a deformation of a monotone Lagrangian $L$ off itself in the normal direction $J\nabla\sigma$, where $\sigma$ is a submersion $L\to S^1$ representing the Maslov class and $J$ is a compatible almost complex structure, is again monotone with smaller monotonicity constant. Applying Theorem \ref{mainthm-unlink}, we deduce that the pushoff is nullhomologous in the complement of $L$. From this we deduce that $J\nabla\sigma$ is the Haefliger-Hirsch field.

By performing some homotopy computations inspired by \cite{Bor} we also derive:

\begin{maincor}\label{maincor-unknot}
If $n\geq 5$ is odd then all monotone Lagrangian $n$-tori in $\CC^n$ are smoothly isotopic.
\end{maincor}

First note that this is true at the level of maps: we do not need to repara\-metr\-ise the tori to make them isotopic. Also note that this is really a rigidity theorem for Lagrangians and the proof uses hard tools (pseudoholomorphic curves) in an essential way: if we were to relax the Lagrangian condition to totally real then there are totally real embeddings of tori representing all isotopy classes of smooth embeddings (by the h-principle for totally real embeddings) and the isotopy classes of smooth embeddings are in bijection with $H_1(L;Y(n))$ where
\[Y(n)=\begin{cases}
\ZZ&\mbox{ if }n\equiv 1\mod 2\\
\ZZ/2&\mbox{ if }n\equiv 0\mod 2
\end{cases}\]
(see Section \ref{sect-HaeHir}). There are many examples of Hamiltonian non-isotopic monotone Lagrangian tori in high-dimensional symplectic vector spaces due to Chekanov and Schlenk \cite{ChekSch} which are, reassuringly, known to be smoothly (in fact Lagrangian) isotopic. In Section \ref{sect-knots} we will construct some examples of relatively knotted monotone tori with non-homotopic Gauss maps when $n$ is even. These are similar to the smoothly knotted Lagrangian $S^1\times S^3$ examples found in \cite{Bor}.

Both of these theorems will be proved in greater generality below, see Theorem \ref{unlink} and Theorem \ref{unknot} respectively. In particular, we will not require $L_1$ and $L_2$ to be tori, but the torus is the simplest manifold satisfying the topological conditions we require. Indeed for Theorem \ref{mainthm-unlink} we do not even require that the two Lagrangians are diffeomorphic. We prove Corollary \ref{maincor-unknot} in Section \ref{sect-cor-prf}.

In Section \ref{sect-gen} we will prove a harder result (Theorem \ref{unknot-tweak}):
\begin{mainthm}\label{mainthm-unknot-tweak}
Suppose $L_1$ and $L_2$ are embedded Lagrangian copies of $S^1\times S^{n-1}$ in $\CC^{n}$ and suppose that they are homotopic through Lagrangian immersions. If
\begin{itemize}
\item $n=2k+2>4$ and $L_1$, $L_2$ have minimal Maslov number $2k+2$ or
\item $n=2k+1>4$ and $L_1$, $L_2$ are monotone
\end{itemize}
then they are smoothly isotopic through embeddings (not necessarily Lagrangian embeddings).
\end{mainthm}
This result was proved for $n=4,8$ by Borrelli \cite{Bor}, even without the assumption on the minimal Maslov number, using completely different methods. It is mysterious to us that our technique cannot deal with the case $n=4$, but this restriction is needed to rule out certain bad limiting behaviour of punctured holomorphic curves under neck-stretching. We will also prove:
\begin{maincor}\label{maincor-sphere-unknot}
Let $n$ be an integer, $n\geq 5$. If $n$ is odd then all monotone Lagrangian embeddings of $S^1\times S^{n-1}$ in $\mathbf{C}^n$ are smoothly isotopic after reparametrisation. If $n$ is even then all Lagrangian embeddings of $S^1\times S^{n-1}$ in $\mathbf{C}^n$ with minimal Maslov number $n$ are smoothly isotopic.
\end{maincor}
\begin{rmk}
When $n$ is odd it may be necessary to reparametrise $S^1\times S^{n-1}$ by a reflection of $S^1$. This is not necessary for the case $n$ even: as Borrelli argues in {\cite[Lemma G]{Bor}}, when $n$ is even the action of diffeomorphisms on the space of embeddings is identified with the usual action of diffeomorphisms on $H^{n-1}(S^1\times S^{n-1};\ZZ/2)=\ZZ/2$ and hence trivial.
\end{rmk}

\section{Basic definitions}
Let $L$ be an oriented $n$-manifold. Equip $L$ with a Riemannian metric and let $\OP{Fr}^+(L)$ denote the bundle of positive orthonormal frames on $L$. Let $(X,\omega)$ be a symplectic manifold equipped with a compatible almost complex structure $J$ and corresponding almost K\"{a}hler metric $g$. Let $\OP{Fr}_U(X)$ denote the bundle of unitary frames on $X$.
\begin{dfn}
An immersion $f\colon L\looparrowright X$ is called Lagrangian if $f^*\omega=0$.
\end{dfn}
Given a Lagrangian immersion $f$, equip $L$ with the metric $f^*g$. One defines a {\em Lagrangian frame map}
\[\OP{Fr}(f)\colon\OP{Fr}^+(L)\to\OP{Fr}_U(X)\]
which sends an orthonormal frame of $T_pL$ to the pushed-forward frame considered as a unitary basis of $T_{f(p)}X$. Note that the $\RR$-span of a unitary basis of $T_xX$ is a Lagrangian $n$-plane and that all Lagrangian $n$-planes arise this way. In particular, $U(n)$ acts transitively on the {\em oriented Lagrangian Grassmannian} $\Lambda^+(n)$ of Lagrangian $n$-planes and the stabiliser is $SO(n)$, acting in the usual way on orthonormal frames in a fixed Lagrangian $n$-plane; so $\Lambda^+(n)\cong U(n)/SO(n)$. If we define $\Lambda(X)$ to be the {\em Lagrangian Grassmann bundle} of all Lagrangian $n$-planes in tangent spaces of $X$ then we see $\OP{Fr}_U(X)$ is an $SO(n)$-bundle over $\Lambda(X)$ and the map $\OP{Fr}{f}$ defined above is an $SO(n)$-equivariant bundle map living over the {\em Lagrangian Gauss map}
\[\Lambda(f)\colon L\to\Lambda(X)\]
which sends $p\in L$ to $f_*(T_pL)\in\Lambda(X)_{f(p)}$.
\begin{thm}[Gromov \cite{Gro}, Lees {\cite[Theorem 1 and Corollary 2.2]{Lee}}]\label{thm-gro-lee}
Two Lagrangian immersions
\[f_1,f_2\colon L\looparrowright X\]
are homotopic through Lagrangian immersions if and only if the maps
\[\OP{Fr}(f)_1,\OP{Fr}(f)_2\colon\OP{Fr}^+(L)\to\OP{Fr}_U(X)\]
are homotopic as $SO(n)$-equivariant bundle maps. Moreover, given an $SO(n)$-equivariant bundle map $\Phi\colon\OP{Fr}^+(L)\to\OP{Fr}_U(X)$ such that the underlying map $\phi\colon L\to X$ satisfies
\[[\phi^*\omega]=0\in H^2(L;\RR),\]
there exists a Lagrangian immersion $F\colon L\looparrowright X$ homotopic to $\phi$ with $\OP{Fr}(F)$ homotopic to $\Phi$.
\end{thm}
Mostly we consider Lagrangian submanifolds in the standard symplectic vector space $\CC^n$. The tangent bundle of $\CC^n$ is canonically trivialised by translation of vectors to the origin. We write
\begin{equation}\label{eq-triv}P_x\colon T_x\CC^n\to T_0\CC^n\end{equation}
for this translation map. In par\-tic\-u\-lar we may iden\-ti\-fy ($SO(n)$-equi\-var\-iant\-ly) the Lagrangian Grassmann bundle $\Lambda(\CC^n)$ with the product
\[\Lambda(\CC^n)\cong\Lambda^+(n)\times\CC^n.\]
In this trivialisation we consider the Lagrangian Gauss and frame maps as maps
\[\OP{Fr}(f)\colon\OP{Fr}^+(L)\to U(n),\quad\Lambda(f)\colon L\to\Lambda^+(n).\]
Note that an $SO(n)$-equivariant bundle map $\OP{Fr}^+(L)\to U(n)$ is a section of the associated $U(n)$-bundle, which is the complexified frame bundle $\OP{Fr}_{\CC}(L)$, a principal $U(n)$-bundle. This means that $L$ admits a Lagrangian immersion in $\CC^n$ if and only if its complexified tangent bundle is trivial and that two Lagrangian frame maps are $SO(n)$-equivariantly homotopic if and only if the corresponding trivialisations of the complexified tangent bundle are homotopic. The difference between two trivialisations of the complexified tangent bundle comprises a map $L\to U(n)$ and hence $SO(n)$-equivariant homotopy classes of Lagrangian frame map correspond (non-canonically) one-to-one with homotopy classes of maps $L\to U(n)$.

It is well-known that $H^1(\Lambda^+(n);\ZZ)\cong\ZZ$ and that this cohomology group is generated by the {\em Maslov class} $\mu$ \cite{Arn}.
\begin{dfn}\label{dfn-monotone}
Let $\lambda_0$ be the standard Liouville form
\[\lambda_0=\sum_{k=1}^nx_kdy_k\]
on $\CC^n$ and $\omega_0=d\lambda_0$ the standard symplectic 2-form. Let $f\colon L\looparrowright\CC^n$ be a Lagrangian immersion. The {\em symplectic area class of $f$} is the cohomology class
\[a(f):=[f^*\lambda_0]\in H^1(L;\RR).\]
The {\em Maslov class of $f$} is the cohomology class
\[\mu(f):=f^*\mu\in H^1(L;\RR).\]
We say that $f$ is $K$-monotone if
\[a(f)=K\mu(f)\]
for some $K>0$.
\end{dfn}
More generally if $(X,\omega)$ is a symplectic manifold and $f\colon L\to X$ a Lagrangian submanifold, there is an area homomorphism
\[a(f)\colon H_2(X,L;\ZZ)\to\RR\]
obtained by integrating $\omega$ over relative chains.

Monotone Lagrangians provide a particularly convenient setting for doing Lagrangian Floer theory \cite{Oh,BC}. This is because the area of a disc controls its Maslov index which controls the expected dimension of the moduli space and hence bubbling phenomena which reduce area also reduce expected dimension, so the boundary of a moduli space will have smaller expected dimension than the moduli space itself. There are good reasons to study monotone Lagrangians in their own right.
\begin{thm}[{\cite[Theorem D]{EvaKed}}]
In the Gromov-Lees h-principle one can require that all Lagrangian immersions are monotone.
\end{thm}
Restrictions on embeddings of monotone Lagrangians (as opposed to immersions) are therefore truly rigidity theorems.

Monotone Lagrangians in $\CC^n$ arise naturally in the Lagrangian mean curvature flow as self-similarly contracting solutions \cite{GroSchSmoZeh}. They also exhibit special properties that are not shared by a general Lagrangian submanifold. For instance, orientable monotone Lagrangians in $\CC^3$ are all products $S^1\times\Sigma_g$ {\cite[Theorem B]{EvaKed}}; counterexamples to the nonmonotone version of this statement can be constructed by applying Polterovich connect-sum to remove double points of Lagrangian immersions obtained by an h-principle.

\section{Context}\label{sect-context}

The current techniques for understanding knottedness or linking of Lagrangian submanifolds fall into several categories:
\subsection{Luttinger surgery}
The papers \cite{Bor,EliPol,Lut} use Luttinger surgery on a hypothetical Lagrangian submanifold with specified self-linking or knotting properties to produce an impossible symplectic manifold. These have the drawback that one must study Lagrangians diffeomorphic to $S^1\times S^k$, $k=1,3,7$, for which one can define Luttinger surgery.
\begin{thm}[Luttinger \cite{Lut}]
There exist isotopy classes of embedded tori in $\CC^2$ which do not contain Lagrangian embeddings.
\end{thm}
\begin{thm}[Luttinger \cite{Lut}, Eliashberg-Polterovich \cite{EliPol}, Borrelli \cite{Bor}]\label{thm-haef-hir-lag}
Suppose $k\in\{1,3,7\}$. If $S^1\times S^k\cong L\subset\CC^{1+k}$ is a Lagrangian submanifold, $W\colon D^*_{\rho,g}L\to\CC^{1+k}$ is a Weinstein neighbourhood, $\sigma\colon L\to S^1$ is the projection to the first factor and $L'$ is the image under $W$ of the graph of the 1-form $\rho'd\sigma$ for some $\rho'<\rho$ then $L'$ does not link $L$.
\end{thm}
In the language of Section \ref{sect-HaeHir} below, this theorem identifies the Haefliger-Hirsch field of the Lagrangian embedding and this is enough to determine the smooth knot type when the dimension is sufficiently large. This argument, due to Borrelli, is reproduced in Section \ref{sect-comp-prf} below. Using it, Borrelli observes:
\begin{cor}[Borrelli \cite{Bor}]\label{cor-borrelli}
Suppose $k\in\{3,7\}$. The smooth isotopy class of a Lagrangian submanifold $S^1\times S^k\subset\CC^{k+1}$ is determined by the homotopy class of the Lagrangian frame map. In other words, two Lagrangian $S^1\times S^k\subset\CC^{k+1}$ with homotopic Lagrangian frame maps are not relatively knotted.
\end{cor}
\subsection{Perturbing the symplectic form}
The papers \cite{EliPol4,HinIvr} perturb the symplectic structure to make a two-dimensional Lagrangian into a symplectic submanifold, choose an almost complex structure making this submanifold into a holomorphic curve and then perturb the almost complex structure to perform isotopies of the holomorphic curve. These techniques have the disadvantage that one must work in very special four-dimensional situations.
\begin{thm}[{\cite[Genus 0, 1]{EliPol4}}, {\cite[Genus >1]{HinIvr}}]
If $\Sigma$ is an orientable closed surface and $L\subset T^*\Sigma$ is a Lagrangian submanifold homologous to the zero section then $L$ is smoothly isotopic to the zero-section.
\end{thm}
\subsection{Filling with discs}
The paper \cite{EliPol3} proves a `local unknottedness theorem' in four dimensions, showing that all Lagrangian planes asymptotic to a linear Lagrangian plane are unknotted.
\begin{thm}
A Lagrangian embedding $\RR^2\to\CC^2$ which agrees with the embedding of a linear Lagrangian plane $\Pi$ outside a compact set is isotopic (via an ambient compactly supported Hamiltonian isotopy) to $\Pi$.
\end{thm}
The technique of proof is to construct an ambient submanifold of codimension one containing the Lagrangian via the method of filling by holomorphic discs. The submanifold thus constructed is also required to satisfy certain conditions on its characteristic foliation. An isotopy of the Lagrangian with a plane is then constructed explicitly.
\subsection{Holomorphic foliations}
Hind's papers \cite{Hin,Hin2} use symplectic field theory and neck-stretching arguments to put Lagrangian spheres into a special position with respect to a pseudoholomorphic foliation. Neck-stre\-tch\-ing arguments were also used in \cite{BorLiWW,Eva,LiWW} to disjoin a Lagrangian sphere from a fixed collection of symplectic submanifolds and reduce various knotting problems to those studied by Hind or to problems on connectivity of spaces of symplectic ball packings. When these methods work they produce very strong results, but the drawback is that they require foliations by holomorphic curves, and such foliations are only well-behaved in dimension four. They also work best for Lagrangian spheres and it has proved difficult to approach Lagrangian tori this way.

\subsection{This paper}
Our methods are completely different in character and rely on the existence of many holomorphic discs with boundary on one of the Lagrangians to produce a nullhomology of that Lagrangian. A neck-stretching argument is used to ensure that this nullhomology can be made disjoint from the other Lagrangian.

Related work on a Floer theoretic approach to homological inclusion maps for Lagrangian submanifolds is discussed in \cite{Alb,AlbErr}.

\part{Proofs}

\section{Holomorphic discs which avoid a Lagrangian}
If $(X,L)$ is a pair consisting of a symp\-lec\-tic amb\-ient ma\-ni\-fold and a Lag\-ran\-gian submanifold, we denote by $\OP{Hur}_*(X,L;\ZZ)$ the image of the Hurewicz homomorphism in relative homology. Let $D^*_{\rho,g}L$ denote the radius $\rho$ closed disc subbundle of the cotangent bundle $T^*L$ for a metric $g$ and $S^*_{\rho,g}L$ its boundary.

\begin{dfn}
Define the {\em infimal disc area} of a Lagrangian embedding $f\colon L\to X$ to be
\begin{equation}\label{eqn-area-inf}A(f)=\inf\left\{a(f)(\beta)\ \middle|\ \beta\in H_2(X,L;\ZZ),\ a(f)(\beta)>0\right\}\end{equation}
where $a(f)\colon H_2(X,L;\ZZ)\to\RR$ is the area homomorphism.
\end{dfn}
\begin{thm}\label{thm-avoid}
Let $L_1$ and $L_2$ be manifolds and suppose that $L_2$ admits a metric $g$ with no contractible geodesics. Suppose that
\begin{align*}
f_1&\colon L_1\to X\\
f_2&\colon L_2\to X
\end{align*}
are Lagrangian embeddings into a symplectic manifold $(X,\omega)$ (either closed or convex at infinity). Suppose that there exists a class $\beta\in \OP{Hur}_2(X,L_1;\ZZ)$ with $a(f_1)(\beta)\leq A(f_2)$.

Then there exists an almost complex structure $J$ on $X$ with the following property. All genus zero $J$-holomorphic curves representing the class $\beta$, that is to say spheres passing through $L_1$ or discs with boundary on $L_1$, are disjoint from $L_2$. Moreover, if $a(f_1)(\beta)=A(f_1)$ then $J$ can be chosen to be regular for the moduli problem of finding discs in the class $\beta$.
\end{thm}

The construction of $J$ proceeds by stretching the neck along a contact-type hypersurface and analysing the limits of holomorphic curves as the neck becomes infinitely long. This is a standard trick but it is important enough that we recall it now. For further details of neck-stretching and SFT compactness we refer to the papers \cite{BEHWZ} and \cite{CieMoh}.

Let
\[W\colon D^*_{\rho,g}L_2\to X\]
be a symplectic embedding, extending $f_2$, given by Weinstein's neighbourhood theorem. Let $0<\rho'<\rho$ and let $\alpha$ be the contact form on $M=S_{\rho',g}^*L_2$ (given by minus the pullback of the Liouville form). The Reeb flow of $\alpha$ is precisely the cogeodesic flow of $g$ and we are assuming there are no contractible closed geodesics. There is an $\epsilon>0$ and a symplectic embedding of
\[\left((-\epsilon,\epsilon)\times M,d(e^r\alpha)\right)\]
into $X$ as a collar neighbourhood of $W(M)$. Here $r$ denotes the coordinate on the interval $(-\epsilon,\epsilon)$. Let $J_0$ be an almost complex structure on $W(D^*_{\rho,g}L_2)$ which preserves the contact structure, is compatible with $d(e^r\alpha)$, sends $\partial_r$ to the Reeb direction and is $r$-invariant on the collar neighbourhood. We will write $\iota\colon(-\epsilon,\epsilon)\times M\to X$ for the embedding of the collar neighbourhood and $U_{\epsilon}$ for its image.

Now we construct a neck-stretching sequence $J_t$ as in {\cite[Section 2.7]{CieMoh}}. Let $X_0=X\setminus U_{\epsilon/2}$ and let $X_k$ be the union of $X_0$ with $N_k=(-\epsilon-k,\epsilon)\times M$, identifying $(-\eta-k,s)\in N_k$ with $\iota(-\eta,s)\in X_0$ and $(\eta,s)\in N_k$ with $\iota(\eta,s)\in X_0$ for $\epsilon/2<\eta<\epsilon$. The almost complex structure $\iota^*J_0$ extends uniquely to an $r$-invariant almost complex structure $I_k$ on $N_k$ and we define $J_k$ to be equal to $J_0$ on $X_0$ and equal to $I_k$ on $N_k$. We equip $X_k$ with a symplectic form $\omega_k$ as follows. Decompose $X_0=W\cup V$ into its connected components, $W$ containing $L_2$ and $V$ containing $L_1$. We define $\omega_k$ to be equal to $\omega$ on $V$, equal to $d(e^r\alpha)$ on $N_k$ and equal to $e^{-k}\omega$ on $W$. The manifolds $X_k$ are all diffeomorphic to one another and the cohomology class of $\omega_k$ is independent of $k$, so by Moser's theorem $(X_k,\omega_k)$ is symplectomorphic to $(X,\omega)$. Pulling back $J_k$ along this symplectomorphism gives us our neck-stretching sequence. For notational convenience we will still denote this almost complex structure by $J_k$.

As $k\to\infty$ the manifolds $(-k-\epsilon,k+\epsilon)\times M$, $W_k=W\cup ((-\epsilon,k+\epsilon)\times M)$ and $V_k=V\cup N_k$ have as direct limits noncompact almost complex manifolds $\overline{S}$, $\overline{W}$ and $\overline{V}$. If we use the symplectic form $d(e^r\alpha)$ on each end of the form $(a,b)\times M$ then the limits $\overline{S}$ and $\overline{V}$ are naturally symplectic; if we rescale the symplectic form on $W_k$ by $e^{-k}$ then the limit is also naturally symplectic. We can identify these limits: $\overline{S}$ is the symplectisation of $M=S_{\rho',g}^*L_2$, $\overline{W}$ is the cotangent bundle of $L_2$ and $\overline{V}$ is the complement $X\setminus L_2$.

The SFT compactness theorem \cite{BEHWZ,CieMoh} states that for a sequence $k_i\to\infty$ and a sequence of $J_{k_i}$-holomorphic curves $u_i\colon\Sigma\to X_{k_i}$ representing some fixed homology class, there is a subsequence $k_{i_j}$ and reparametrisations $\phi_{i_j}$ of $\Sigma$ such that $u_{i_j}\circ\phi_{i_j}^{-1}$ converge (in the Gromov-Hofer sense) to a {\em holomorphic building} which is a union
\[u=u_{\overline{W}}\cup u_{\overline{S}_1}\cup \cdots\cup u_{\overline{S}_{\ell}}\cup u_{\overline{V}}\]
of punctured finite-energy holomorphic curves, where $u_{\overline{W}}$ is a curve in $\overline{W}$, $u_{\overline{V}}$ is a curve in $\overline{V}$ and each $u_{\overline{S}_i}$ is a curve in $\overline{S}$. These curves are asymptotic to cylinders on Reeb orbits and, if we were to glue all components together along their common asymptotics we would end up with a two-dimensional topological surface homeomorphic to $\Sigma$.

Though this theorem is proved only for closed surfaces, the proof extends to the case of holomorphic curves with boundary on a Lagrangian submanifold {\cite[Section 11.3]{BEHWZ}}; indeed, in our particularly simple case the Lagrangian is itself compact.

The key property of Gromov-Hofer convergence which we require for our theorem is the behaviour of area in the limit. We state the relevant result from \cite{CieMoh} in our notation:
\begin{lma}[{\cite[Corollary 2.11]{CieMoh}}]\label{lma-area-conserved}
In the above setting, suppose that $u_k\colon\Sigma\to X_k$ is a Gromov-Hofer convergent sequence of $J_k$-holomorphic maps whose limit building is
\[u=u_{\overline{W}}\cup u_{\overline{S}_1}\cup \cdots\cup u_{\overline{S}_{\ell}}\cup u_{\overline{V}}.\]
Let $\Sigma'$ denote the domain of $u_{\overline{V}}$ and let $\omega_{\overline{V}}$ denote the symplectic form on $\overline{V}$. Then
\[\lim_{k\to\infty}\int_{\Sigma}u_k^*\omega_k=\int_{\Sigma'}u_{\overline{V}}^*\omega_{\overline{V}}.\]
Note that $(\overline{V},\omega_{\overline{V}})$ is symplectomorphic to $(X\setminus L_2,\omega)$.
\end{lma}
Now we will prove Theorem \ref{thm-avoid}.
\begin{proof}[Proof of Theorem \ref{thm-avoid}]
Let us deal with the case of discs: the argument for spheres is entirely analogous. If the theorem is false then for all $t\in[0,\infty)$ there exists a $J_t$-holomorphic disc $u_t\colon (D^2,\partial D^2)\to (X,L_1)$ with boundary on $L_1$ representing the class $\beta$ such that $u_t(D^2)\cap L_2\neq\emptyset$. By the SFT compactness theorem we can find a sequence $t_i\to\infty$ such that $u_i$ Gromov-Hofer converges (after reparametrisations) to a holomorphic building $u=u_{\overline{W}}\cup u_{\overline{S}_1}\cup\cdots\cup u_{\overline{S}_{\ell}}\cup u_{\overline{V}}$. Note that a punctured finite-energy curve in $\overline{V}\cong X\setminus L_2$ is asymptotic to a collection of geodesics in $L_2$ and we can compactify it to get a compact topological surface with boundary on $L_2$.

\begin{figure}
\includegraphics[scale=0.4]{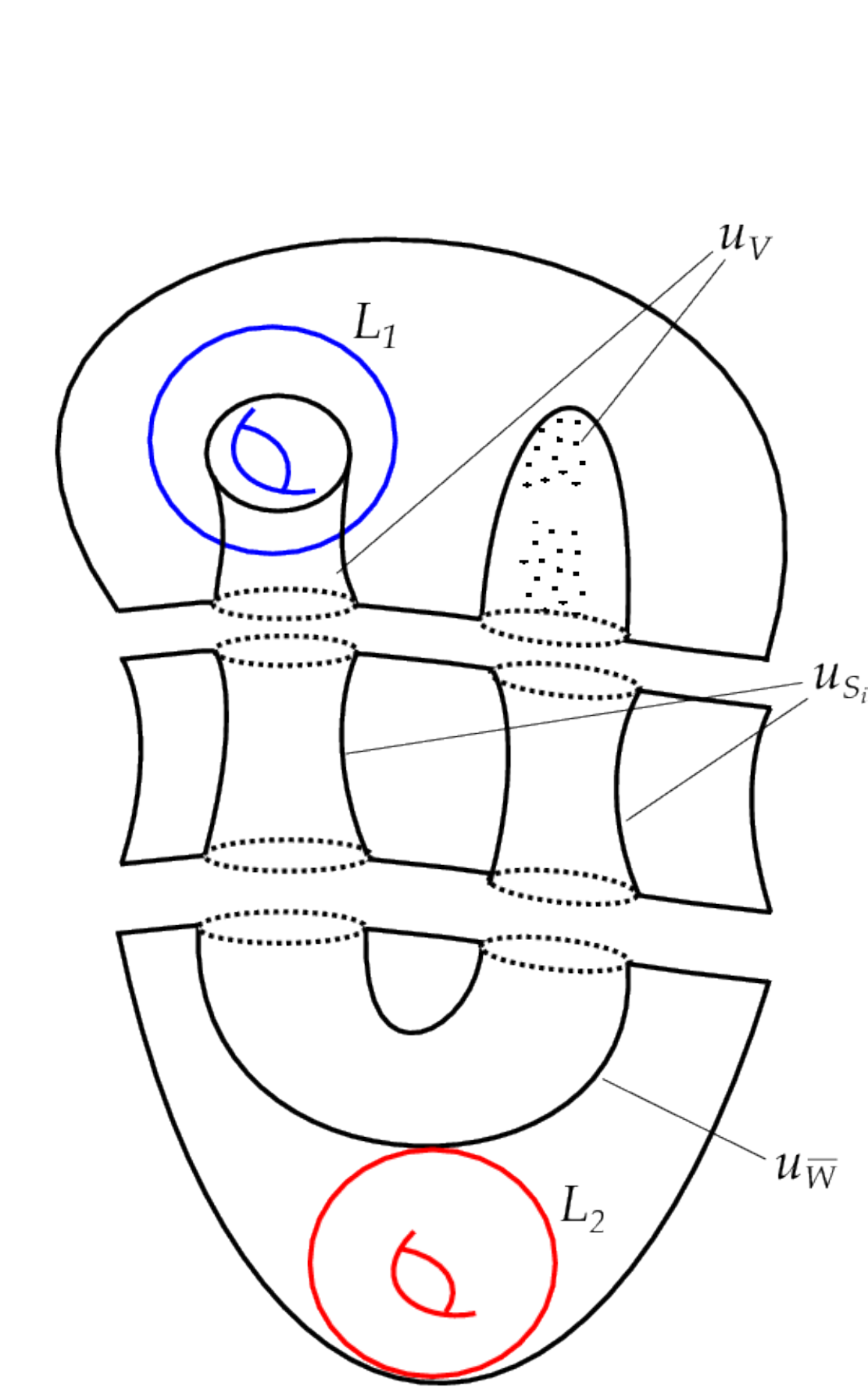}
\caption{The Gromov-Hofer limit must contain a finite-energy plane as a component of $u_V$. This plane (shaded in the figure) has large area since it can be considered as a disc with boundary on $L_2$. This contradicts the conservation of area in the limit.}
\end{figure}

Since the domain of $u_t$ is a disc there must exist a component $v$ of the building whose domain is a plane. The components of the building are asymptotic to closed Reeb orbits. Since there are no contractible Reeb orbits, $v$ has image in $\overline{V}\cong X\setminus L_2$. The asymptotic Reeb orbit of $v$ is a loop in $L_2$ and hence we can think of $v$ as a topological disc with boundary on $L_2$. Its symplectic area is therefore at least $A(f_2)$. Moreover there is another component of $u_{\overline{V}}$ with boundary on $L_1$. However, by Lemma \ref{lma-area-conserved}, the areas of the components of $u_{\overline{V}}$ sum to give $a(f_1)(\beta)$. Since these summands all have strictly positive area we see that $a(f_1)(\beta)>A(f_2)$, which contradicts our assumption that $a(f_1)(\beta)\leq A(f_2)$.

We now show that we can choose $J$ regular for the moduli problem if $a(f_1)(\beta)=A(f_1)$. With this extra hypothesis, $\beta$ is a minimal homology class of discs so we know that pseudoholomorphic discs representing $\beta$ are somewhere injective (by {\cite[Theorem A]{Laz}}) and hence we can achieve transversality for a Baire set of almost complex structures. Let $J_k$ be a sequence of almost complex structures from this Baire set such that $J_k\to J$ as $k\to\infty$ where $J$ is the almost complex structure we have already constructed. Suppose that for all $k$ is a $J_k$-disc $u_k$ representing $\beta$ with $u_k(D^2)\cap L_2\neq\emptyset$. There is a Gromov convergent subsequence $u_{k_j}$. Since $\beta$ is minimal there is no bubbling and the limit is a $J$-disc or sphere representing $\beta$ and intersecting $L_2$, which is a contradiction.
\end{proof}

\section{Unlinking of Lagrangians}

\begin{thm}\label{unlink}
Suppose that $L_1$ and $L_2$ are $n$-manifolds such that:
\begin{itemize}
\item $L_2$ admits a metric with no contractible geodesics,
\item $H^{n-1}(L_2;\ZZ)$ is torsionfree,
\item the cohomology of the universal cover of $L_1$ vanishes in odd degrees,
\item $L_1$ is orientable and spin.
\end{itemize}
Let $f_k\colon L_k\to\CC^n$, $k=1,2$, be Lagrangian embeddings and suppose that
\begin{itemize}
\item $L_1$ is monotone,
\item $L_1$ links $L_2$.
\end{itemize}
Then
\[A(f_2)<A(f_1),\]
where $A(f_k)$ are defined by Equation \eqref{eqn-area-inf}.
\end{thm}
\begin{proof}
By a theorem of Damian \cite{Dam} and its sharpening {\cite[Proposition 7]{EvaKed}} we know that for a regular $J$ there is a free homotopy class $\beta$ of loops on $L_1$ with Maslov index 2 such that, when $J$ is regular, the moduli space
\[\mM_{0,1}(\beta,J)\]
(of $J$-holomorphic discs with boundary on $L_1$ representing $\beta$ and having one boundary marked point) contains a component $M$ such that the evaluation map
\[\ev\colon M\to L_1\]
has nonzero degree, say $d$. Note that the homology class of $\beta$ is minimal by monotonicity. For contradiction, assume that $A(f_2)\geq A(f_1)$. By Theorem \ref{thm-avoid} we may chose $J$ so that no discs in $M$ intersect $L_2$. Let $N$ denote the moduli space of $J$-holomorphic discs from $M$ with a marked point on the interior. Note that $\partial N=M$. Then $N$ gives a nullhomology of $[M]=d[L_1]$ in $\CC^n\setminus L_2$. Since $H_n(\CC^n\setminus L_2,\ZZ)\cong H^{n-1}(L_2;\ZZ)$ is torsionfree by assumption, this implies that $L_1$ is nullhomologous in $\CC^n\setminus L_2$.
\end{proof}

For monotone Lagrangian tori one does not require the full lifted Floer homology used by Damian or the proposition from \cite{EvaKed} to get the existence of these holomorphic discs; one can get away with an argument, due to Buhovsky, using the quantum product structure on Floer cohomology and the Oh spectral sequence: see {\cite[Theorem 2]{Buh}}.

\section{Unknottedness of Lagrangians}

\subsection{Statement of results}

\begin{thm}\label{unknot}
Let $n\geq 4$ and let $L$ be an $n$-manifold satisfying the following conditions:
\begin{enumerate}
\item\label{condition-contractible-geodesics} $L$ admits a metric $g$ with no nonconstant contractible geodesics,
\item\label{condition-torsionfree} the cohomology group $H^{n-1}(L;\ZZ)$ is torsionfree.
\end{enumerate}
Suppose that $f_1\colon L\to\CC^n$ is a monotone Lagrangian embedding and suppose that there exists a submersion $\sigma\colon L\to S^1$ with connected fibre satisfying
\[\mu(f_1)=\sigma^*\left[\frac{d\theta}{2\pi}\right].\]
Suppose that $f_2\colon L\to\CC^n$ is another monotone Lagrangian embedding such that the Lagrangian frame maps are $SO(n)$-equivariantly homotopic when restricted to the complement $L_x$ of a point $x\in L$:
\[\OP{Fr}(f_1)|_{L_x}\simeq_{SO(n)}\OP{Fr}(f_2)|_{L_x}.\]
Then $f_2$ is smoothly iso\-to\-pic to $f_1$ through embeddings (not necessarily through Lagrangian embeddings).
\end{thm}
Note that the frame maps being equivariantly homotopic when restricted to the complement of a point is a weaker condition than being equivariantly homotopic over the whole of $L$. We used the latter condition in the statements of Theorems \ref{mainthm-unknot} and \ref{mainthm-unknot-tweak} to make them more readable, but only the weaker assumption is needed.
\begin{rmk}\label{rmk-asph}
Note that the condition that $L$ admits a metric $g$ with no nonconstant contractible geodesics implies that $L$ is aspherical. This is a theorem of Lyusternik and Fet \cite{LyuFet}: if there is no nonconstant contractible geodesic then the space of contractible based loops retracts to the constant loop via the downward gradient flow of the energy functional; since this component of the based loopspace is contractible, all higher homotopy groups vanish.
\end{rmk}

\begin{cor}\label{cor-unknot}
For $n\geq 4$, the smooth isotopy class of a monotone Lagrangian embedding of the $n$-torus in $\CC^n$ is determined by the $SO(n)$-equivariant homotopy type of the Lagrangian frame map.
\end{cor}
\begin{proof}
The flat $n$-torus, $n\geq 4$, satisfies Conditions \ref{condition-contractible-geodesics} and \ref{condition-torsionfree} of Theorem \ref{unknot}. If $f_1\colon T^n\to\CC^n$ is a monotone Lagrangian embedding then by \cite{Buh} or \cite{Dam} we know there is a relative homology class $\beta\in H_2(\CC^n,T^n;\ZZ)\cong H_1(L;\ZZ)$ with $\mu(\beta)=2$. In particular, half the Maslov class $\frac{1}{2}\mu(f_1)\in H^1(L;\ZZ)$ is a primitive cohomology class: if $\mu(f_1)=k\nu$ for some class $\nu\in H^1(L;\ZZ)$ then $1=\frac{1}{2}\mu(\beta)=k\nu(\beta)$ so $k=\pm 1$. Since $L$ is a torus, any primitive homology class $c\in H_1(L;\ZZ)$ can be represented as $\sigma^*\left[\frac{d\theta}{2\pi}\right]$ for a circle-valued function $\sigma$ satisfying the conditions of Theorem \ref{unknot}. The corollary is now immediate.
\end{proof}
\begin{rmk}
By Theorem \ref{thm-gro-lee}, knowing the $SO(n)$-equivariant homotopy type of the Lagrangian frame map is the same as knowing the homotopy class of Lagrangian immersions. Therefore Theorem \ref{mainthm-unknot} is a direct consequence of Corollary \ref{cor-unknot}.
\end{rmk}

\subsection{Haefliger-Hirsch theory}\label{sect-HaeHir}

In their paper \cite{HaeHir} Haefliger and Hirsch classified smooth embeddings of compact $n$-manifolds into $\CC^n$ up to isotopy. We explain their results in this section.

\begin{dfn}\label{dfn-pushoff}
Let $L$ be a closed, oriented, connected $n$-manifold and suppose that $f\colon L\to\CC^n$ is an embedding. Let $x\in L$ be a point, let $x\in D'\subset D\subset L$ be nested disc neighbourhoods of $x$; let $\rho$ be a cut-off function supported on $L\setminus D'$ and equal to one outside $D$. Define $L_x=L\setminus\{x\}$ and $f_x=f|_{L_x}$. If $v\in\Gamma(f_x^*(T\CC^n))$ is a vector field on $L_x$, normal to $L_x$, then $\rho v$ extends to a vector field on $L$, vanishing on $D'$. Fix a tubular neighbourhood of $L$, identified via the exponential map with the normal bundle $\nu L$. Then for sufficiently small $\epsilon$ we can consider the section $\epsilon\rho v$ of $\nu L$ as a submanifold of $\CC^n$ and we call this the {\em pushoff $L'$ of $L$ along $v$}. Note that $L'\subset \CC^n\setminus(L\setminus D)$.
\end{dfn}
\begin{dfn}
In the context of the previous definition, a {\em Haefliger-Hirsch field}, $v\in \Gamma(f_x^*(T\CC^n))$, is a normal vector field along $L_x$ with the property that the pushoff of $L$ along $v$ is nullhomologous in $\CC^n\setminus (L\setminus D)$.
\end{dfn}
\begin{lma}[See \cite{HaeHir}]
Given an embedding $f$ there is a Haefliger-Hirsch field $v$ which is unique up to homotopy.
\end{lma}
Define a fibre bundle
\[\mathcal{V}_{2n,n+1}\to L\]
whose fibre is the Stiefel manifold $V_{2n,n+1}\cong SO(2n)/SO(n-1)$ of $n+1$-frames in $\RR^{2n}$. This bundle is associated to the $SO(n)$-frame bundle of $L$ by the action of $SO(n)$ on the first $n$ vectors of the $n+1$-frame. Given an embedding $f$ with Haefliger-Hirsch field $v$ we can define a section $s(f)$ of $\mathcal{V}_{2n,n+1}|_{L_x}$. Giving a section of this associated bundle is equivalent to giving an $SO(n)$-equivariant map from the oriented frame bundle $\OP{Fr}^+(L_x)$ to the Stiefel manifold. We just send an oriented orthonormal frame $F$ in $T_yL$ to
\[s(f)(y)=(P_{f(y)}f_*(F),P_{f(y)}v)\in V_{2n,n+1}.\]
where $P$ denotes the trivialisation defined by Equation \eqref{eq-triv}. We are writing the $n+1$-frame as $(G,w)$ where $G$ consists of the first $n$ vectors and $w$ is the last. We call the $SO(n)$-equivariant map
\[s(f)\colon\OP{Fr}^+(L_x)\to V_{2n,n+1}=SO(2n)/SO(n-1)\]
the {\em Haefliger-Hirsch map}.
\begin{thm}[See \cite{HaeHir}]\label{thm-HaeHir}
Suppose $n\geq 4$. If $f'\colon L\to\CC^n$ is another embedding then it is smoothly isotopic to $f$ if and only if the Haefliger-Hirsch maps $s(f)$ and $s(f')$ are homotopic as $SO(n)$-equivariant maps $\OP{Fr}^+(L_x)\to V_{2n,n+1}=SO(2n)/SO(n-1)$.
\end{thm}
If $\Gamma(\mathcal{V}_{2n,n+1}|_{L_x})$ denotes the space of sections of the Stiefel manifold bundle over $L_x$ (i.e. the space of Haefliger-Hirsch maps) Haefliger and Hirsch identify
\begin{align*}\pi_0(\Gamma(\mathcal{V}_{2n,n+1}|_{L_x}))&=H^{n-1}(L;\pi_{n-1}(V_{2n,n+1}))\\
&=\begin{cases}H_1(L;\ZZ)&\mbox{ if }n\equiv 1\mod 2\\
H_1(L;\ZZ/2)&\mbox{ if }n\equiv 0\mod 2.\end{cases}\end{align*}
This is an application of obstruction theory and of the fact that $V_{2n,n+1}$ is $n-2$-connected and $\pi_{n-1}(V_{2n,n+1})=\ZZ$ (respectively $\ZZ/2$) when $n$ is odd (respectively even); see {\cite[Theorem 3.16]{Pras}} for this latter computation.

\subsection{Identifying the Haefliger-Hirsch field}

Take a metric on $L$ and let $\sigma\colon L\to S^1$ be a submersion representing the Maslov class in the sense that $\mu(f_1)=\sigma^*\left[\frac{d\theta}{2\pi}\right]$. The gradient of $\sigma$ is a nowhere-vanishing vector field $\nabla\sigma$ and under the musical isomorphism with the cotangent bundle it is identified with the 1-form $d\sigma$. Let $J$ be a compatible almost complex structure on $T^*L$ and consider $\nabla\sigma$ as a vector field defined along the zero-section. 

Let $\{L_t\colon L\to T^*L\}_{t\in(-1,1)}$ be a 1-parameter family of embedded submanifolds such that
\begin{itemize}
\item $L_0$ is just the inclusion of the zero-section,
\item $\left.\frac{d}{dt}\right|_{t=0}L_t(y)=J(L_0)_*\nabla\sigma(y)$ for all $y\in L$,
\item $\OP{im}(L_t)\cap \OP{im}(L_0)=\emptyset$ for $0\neq |t|<\epsilon$.
\end{itemize}
Then for $\epsilon>0$, $L_{\epsilon}$ is smoothly isotopic in $T^*L\setminus L$ to the inclusion of the graph of the 1-form $d\sigma$.

In the setting of Theorem \ref{unknot}, let $f_1\colon L\to\CC^n$ be a monotone Lagrangian embedding. Let $W\colon D^*_{\rho}L\to\CC^n$ be a Weinstein neighbourhood of $f_1(L)$ in $\CC^n$ and $L'$ be the image of the graph of $\rho' d\sigma$ under $W$ for some $0<\rho'<\rho$.
\begin{lma}\label{lma-monopush}
If $L$ is $K$-monotone with $\mu(f_1)=\sigma^*\left[\frac{d\theta}{2\pi}\right]$ the Lagrangian $L'$ is $K'$-monotone where
\[K'=K-2\pi\rho'<K.\]
\end{lma}
\begin{proof}
Since $L$ and $L'$ are Lagrangian isotopic their Lagrangian Gauss maps are homotopic so the Maslov class is unchanged. We need to compute the symplectic area class
\[[\lambda_0]\in H^1(L';\RR)\cong\OP{Hom}(H_1(L';\RR),\RR).\]
Let $\alpha$ denote minus the (canonical) Liouville 1-form on $D^*_{\rho,g}L$ so that $d\alpha$ is the canonical symplectic form. Note that $W^*\lambda_0-\alpha$ is closed because the Weinstein embedding is symplectic, and its cohomology class is $[\lambda_0]\in H^1(L;\RR)$. Note that
\[H^1(D^*_{\rho,g}L;\RR)\stackrel{i^*}{\cong} H^1(L;\RR)\]
where $i\colon L\to D^*_{\rho,g}L$ is the inclusion, therefore $W^*\lambda_0=\alpha+G+dF$ where $G$ is a form representing $(i^*)^{-1}[\lambda_0]$ and $F$ is a function. Let $i_{\rho'd\sigma}\colon L'\to D^*_{\rho,g}$ be the inclusion of $L'$ into the Weinstein neighbourhood and $f'_1=W\circ i_{\rho'd\sigma}$ the embedding of $L'$ into $\CC^n$. Then
\begin{align*}
[(f'_1)^*\lambda_0]&=[i_{\rho'd\sigma}^*W^*\lambda_0]\\
&=[i_{\rho'd\sigma}^*(\alpha+G+dF)]\\
&=-\rho'[d\sigma]+K\mu(f_1)
\end{align*}
by definition of the canonical 1-form $-\alpha$ and by $K$-monotonicity. Since $[d\sigma]=2\pi\mu(f_1)$ by assumption we get monotonicity with constant
\[K'=K-2\pi\rho'.\]
\end{proof}

\begin{lma}\label{lma-HaeHir}
The vector field $v=J(f_1)_*\nabla\sigma$ restricted to the complement of a point $x\in L$ is a Haefliger-Hirsch field for $f_1$.
\end{lma}
\begin{proof}
As we observed above, the pushoff of $L$ along $v$ is smoothly isotopic to a $K'$-monotone Lagrangian $L'$ where $K'<K$. Since $L$ admits a metric without contractible geodesics, we know that $L$ is aspherical by Remark \ref{rmk-asph}. Therefore $L_2=L$ and $L_1=L'$ satisfy the assumptions of Theorem \ref{unlink}, so $L'$ is nullhomologous in the complement of $L$. Fix a point $x\in L$ and a disc neighbourhood $x\in D\subset L$ and let $v'$ denote the restriction of $v$ to $L_x$. Since $v$ is nowhere-vanishing, the pushoff of $L$ along $v'$ (as constructed in Definition \ref{dfn-pushoff}) is isotopic in $\CC^n\setminus(L\setminus D)$ to $L'$ which we have seen is nullhomologous in $\CC^n\setminus L$ and therefore in $\CC^n\setminus(L\setminus D)$. Therefore $v'$ is a Haefliger-Hirsch field for $L$.
\end{proof}

\subsection{Unknottedness}\label{sect-comp-prf}

Theorem \ref{unknot} follows immediately from Lemma \ref{lma-HaeHir} and the following:

\begin{lma}\label{lma-htpy-HaeHir}
Suppose that $f_1,f_2\colon L\to\CC^n$ are Lagrangian embeddings such that the Maslov class $\mu$ is represented by $\sigma^*\left[\frac{d\theta}{2\pi}\right]$ for a circle-valued function $\sigma$ with no critical points and such that $J(f_i)_*\nabla\sigma$ are the Haefliger-Hirsch fields for some compatible almost complex structure $J$. If the Lagrangian frame maps associated to $f_i$ are homotopic when restricted to $L_x$, then the respective Haefliger-Hirsch maps are homotopic.
\end{lma}
\begin{proof}
Suppose that $\mathcal{F}_t$, $t\in[1,2]$, is a homotopy of $SO(n)$-equivariant Lagrangian frame maps
\[\mathcal{F}_t\colon\OP{Fr}^+(L_x)\to U(n)\]
with $\mathcal{F}_i=\OP{Fr}(f_i)$ for $i=1,2$. For each frame $e=(e_1,\ldots,e_n)\in\OP{Fr}^+(L_x)$ define the cofficients $a(e)_k$ by $\nabla\sigma=\sum_{k=1}^na(e)_ke_k$. Now
\begin{gather*}
s(\mathcal{F}_t)\colon\OP{Fr}^+(L_x)\to V_{2n,n+1}=SO(2n)/SO(n-1),\\
s(\mathcal{F}_t)(e_1,\ldots,e_n)=(\mathcal{F}_te_1,\ldots,\mathcal{F}_te_n,J\sum_{k=1}^na(e)_k\mathcal{F}_t(e_k))
\end{gather*}
is a homotopy of $SO(n)$-equivariant maps connecting the Haefliger-Hirsch maps
\begin{gather*}
s(f_i)\colon\OP{Fr}^+(L_x)\to V_{2n,n+1}=SO(2n)/SO(n-1),\\
s(f_i)(e_1,\ldots,e_n)=((f_i)_*e_1,\ldots,(f_i)_*e_n,J(f_i)_*\nabla\sigma)
\end{gather*}
for $i=1,2$.
\end{proof}

We will now prove a more precise result which will allow us to prove Corollary \ref{maincor-unknot}. Recall from Section \ref{sect-HaeHir} that given two embeddings $f_1,f_2\colon L \to \CC^n$, there is a difference class $\epsilon(f_1,f_2) \in H^{n-1}(L;\pi_{n-1}(V_{2n,n+1}))$ which vanishes if and only if the respective Haefliger-Hirsch maps are homotopic. We begin by showing the following more general statement about this difference class.

Let $U(n)\to SO(2n)/SO(n-1)=V_{2n,n+1}$ be the map induced by the inclusion of $U(n)$ in $SO(2n)$ and let
\[u_n\colon\pi_{n-1}(U(n))\to\pi_{n-1}(V_{2n,n+1})\]
be the induced map on $(n-1)$th homotopy groups.

Suppose that $n$ is even. Let $e_{n/2} \in H^{n-1}(U(n);\ZZ)$ be the characteristic class in $U(n)$ induced by the $\tfrac{n}{2}$th Chern class in $BU(n)$ via suspension, that is by pulling back this Chern class in $H^n(BU(n),\star;\ZZ)$ to a relative class $H^n(EU(n),U(n);\ZZ)$ and mapping it by the inverse of the connecting homomorphism in the long exact sequence of the pair $(EU(n),U(n))$ to $H^{n-1}(U(n);\ZZ)$. Fixing a trivialisation of $\CC \otimes TL$, the Lagrangian frame-map induced by the respective embeddings can be represented by a map $F_i \colon L \to U(n)$. The difference class
\[ \Delta(f_1,f_2):= F_1^*(c_{n/2})-F_2^*(c_{n/2}) \in H^{n-1}(L,\ZZ)\]
is well-defined independently of the choice of trivialisation and vanishes if the Lagrangian frame maps $F_i$ are homotopic.

\begin{prp}
\label{prp-diff-class}
Suppose that $f_1,f_2\colon L\to\CC^n$ are Lagrangian embeddings such that the Maslov class $\mu(f_i)$ is represented by $\sigma_i^*\left[\frac{d\theta}{2\pi}\right]$ for a circle-valued function $\sigma_i$ with no critical points. Suppose that $J(f_i)_*\nabla\sigma_i$ are the Haefliger-Hirsch fields ($i=1,2$). Suppose moreover that $\nabla\sigma_1$ and $\nabla\sigma_2$ are both homotopic through nonvanishing vector fields to some fixed unit vector field $v$, so that $J(f_i)_*v$ are also Haefliger-Hirsch fields ($i=1,2$). Then $\epsilon(f_1,f_2)$ vanishes when $n$ is odd and
\[ \epsilon(f_1,f_2) = \Delta(f_1,f_2) \mod 2\]
when $n$ is even. Furthermore, if the Lagrangian frame maps are homotopic when restricted to the $(n-2)$-skeleton of $L$, then $\epsilon(f_1,f_2)$ vanishes for $n \geq 5$.
\end{prp}

We begin by noting two lemmas.

\begin{lma}\label{lma-frobenius}
Let $L$ be an oriented manifold with a nowhere-vanishing vector field $v$. Denote by $\OP{Fr}^+(L)$ the principal bundle of oriented frames and by $\OP{Fr}^{\uparrow}(L)\subset\OP{Fr}^+(L)$ the principal $SO(n-1)$-bundle of frames whose first vector is $v/|v|$. An $SO(n)$-equivariant map $\OP{Fr}^+(L)\to V_{2n,n+1}$ restricts to an $SO(n-1)$-equivariant map $\OP{Fr}^{\uparrow}(L)\to V_{2n,n+1}$ and an $SO(n-1)$-equivariant map
\[\alpha\colon\OP{Fr}^{\uparrow}(L)\to V_{2n,n+1}\]
extends uniquely to an $SO(n)$-equivariant map
\[\hat{\alpha}\colon\OP{Fr}^+(L)\to V_{2n,n+1}.\]
\end{lma}
\begin{proof}
The map $\hat{\alpha}$ is given by
\[\hat{\alpha}(e)=\phi\alpha(\phi^{-1}e)\]
where $e\in SO(n)=\OP{Fr}_p^+(L)$ is a frame at $p\in L$ and $\phi$ is an element of $SO(n)$ for which $\phi^{-1}e$ has $v/|v|$ as its first vector.
\end{proof}

\begin{lma}\label{lma-htpy}
Suppose that $n\geq 5$. Let $p\colon U(n)\to SO(2n)/SO(n-1)=V_{2n,n+1}$ be the map induced by the inclusion of $U(n)$ in $SO(2n)$. Then the induced map
\[u_n\colon\pi_{n-1}(U(n))\to\pi_{n-1}(V_{2n,n+1})\]
vanishes.
\end{lma}
\begin{proof}
When $n>1$ we have
\[\pi_{n-1}(U(n))=\begin{cases}
0&\mbox{ if }n\equiv 1\mod 2\\
\ZZ&\mbox{ if }n\equiv 0\mod 2
\end{cases}\]
by Bott periodicity. So when $n$ is odd, $u_n$ vanishes automatically.

The map $u_n$ factors as
\[\pi_{n-1}(U(n))\to\pi_{n-1}(SO(2n))\stackrel{\Upsilon_n}{\longrightarrow}\pi_{n-1}(V_{2n,n+1}).\]
The map $\Upsilon_n\colon\pi_{n-1}(SO(2n))\to\pi_{n-1}(V_{2n,n+1})$ lives in the long exact sequence
\[\pi_{n-1}(SO(2n))\stackrel{\Upsilon_n}{\longrightarrow}\pi_{n-1}(V_{2n,n+1})\to\pi_{n-2}(SO(n-1))\to\pi_{n-2}(SO(2n))\]
Assume $n$ is even. Since $\pi_{n-1}(V_{2n,n+1})=\ZZ/2$ we see that $\Upsilon_n$ vanishes if its cokernel is nontrivial. When $n\equiv 2\mod 8$, using the tables on the first and second pages of \cite{Ker}, we get
{\footnotesize \[\begin{CD}
\pi_{n-1}(SO(2n))@>{\Upsilon_n}>>\pi_{n-1}(V_{2n,n+1})@>>>\pi_{n-2}(SO(n-1))@>>>\pi_{n-2}(SO(2n))\\
@| @| @| @|\\
\ZZ @>>> \ZZ/2 @>>> \ZZ/2\times\ZZ/2 @>>>\ZZ/2
\end{CD}\]}

\noindent and it is clear that $\Upsilon_n$ has nontrivial cokernel. Note that when $n$ is not congruent to 2 modulo 8, the rightmost term vanishes by Bott periodicity and hence the cokernel of $\Upsilon_n$ is precisely the group $\pi_{n-2}(SO(n-1))$. From the tables in \cite{Ker}, we see that $\pi_{n-2}(SO(n-1))$ is always nontrivial unless $n=2,4,8$.

When $n=8$, $\pi_7(SO(16))=\ZZ$ by Bott periodicity and the map to $\pi_7(V_{16,9})$ is surjective. However the map $\ZZ=\pi_7(U(8))\to\pi_7(SO(16))=\ZZ$ is multiplication by 2, so the composite $\pi_7(U(8))\to\pi_7(V_{16,9})$ is trivial, which is what we wanted to show.

To see that the map $\ZZ=\pi_7(U(8))\to\pi_7(SO(16))=\ZZ$ is given by multiplication by 2 note that it arises in the exact sequence
\[\cdots\to\pi_7(U(8))\to\pi_7(SO(16))\to\pi_7(SO(16)/U(8))\to\pi_6(U(8))\to\cdots\]
and that $\pi_6(U(8))=0$ by Bott periodicity and $\pi_7(SO(16)/U(8))=\ZZ/2$.
\end{proof}

\begin{proof}[Proof of Proposition \ref{prp-diff-class}]
Consider the sub-fibre bundle $\OP{Fr}^{\uparrow}(L_x) \subset \OP{Fr}^{+}(L_x)$ of the frame bundle which consists of frames whose first basis vector points in the direction of $v$. Observe that this is naturally a principal $SO(n-1)$-bundle.

The Haefliger-Hirsch maps
\begin{gather*}
s(f_i)\colon\OP{Fr}^+(L_x)\to V_{2n,n+1}=SO(2n)/SO(n-1),\\
s(f_i)(e_1,\ldots,e_n)=(e_1,\ldots,e_n,J(f_i)_*v)
\end{gather*}
which are $SO(n)$-equivariant maps, restrict to $SO(n-1)$-equivariant maps
\[s(f_i)^{\uparrow} \colon\OP{Fr}^{\uparrow} (L_x)\to  V_{2n,n+1}.\]

Since the first vector is $e_1=v$, the restrictions $s(f_i)^{\uparrow}$ factorise as
\[ s(f_i)^{\uparrow}=p \circ \OP{Fr}(f_i)^{\uparrow},\]
where
\[\OP{Fr}(f_i)^{\uparrow} \colon \OP{Fr}^{\uparrow}(L_x) \to U(n)\]
is the restriction of the Lagrangian frame map, and where the projection
\[ p \colon U(n) \to SO(2n)/SO(n-1) =V_{2n,n+1}\]
is induced by the inclusion $U(n) \subset SO(2n)$.

Furthermore, for a fixed choice of complex trivialisation of $\CC \otimes TL$, the Lagrangian frame maps are expressed as maps
\[ F_i \colon L \to U(n).\]
Using the $SO(n-1)$-equivariant map $F \colon \OP{Fr}^{\uparrow}(L) \to U(n)$ which identifies a given frame with a matrix representing the complexified frame relative the above trivialisation, we can write
\[\OP{Fr}(f_i)^{\uparrow}=F_i\cdot F.\]
Here $\cdot$ denotes multiplication of $U(n)$-matrixes.

Recall that the spaces $V_{2n,n+1}$ are $(n-2)$-connected and that
\[\pi_{n-1}(V_{2n,n+1})=Y(n)=\begin{cases} \ZZ/2, & n\ \text{even},\\
\ZZ, & n\ \text{odd}.
\end{cases},\]
The Hurewicz isomorphism implies that, $H^{n-1}(V_{2n,n+1};\ZZ)\cong Y(n)$. Fix a generator $g_n\in H^{n-1}(V_{2n,n+1};\ZZ)$.

We start by showing
\[\epsilon(f_1,f_2)=(p \circ F_1)^*(g_n)-(p \circ F_1)^*(g_n).\]

By Lemma \ref{lma-frobenius}, the two $SO(n)$-equivariant maps $s(f_i)$ are $SO(n)$-equi\-var\-ian\-tly homotopic if and only if the $SO(n-1)$-equivariant maps
\[s(f_i)^{\uparrow}=p\circ(F_i\cdot F)=\rho_F(p\circ F_i)\]
are $SO(n-1)$-equivariantly homotopic, where $\rho_M$ denotes the action on $V_{2n,n+1}$ induced by multiplication on the right by a matrix $M$. This happens if and only if the maps $p \circ F_i$ are homotopic. The obstruction to the problem of finding a homotopy between the two maps
\[p \circ F_i \colon L_x \to V_{2n,n+1}\]
is given by $(p \circ F_1)^*(g_n)-(p \circ F_1)^*(g_n)$, which implies the statement.

When $n$ is odd, $H^{n-1}(V_{2n,n+1};\ZZ)\cong\ZZ$ and $H^{n-1}(U(n);\ZZ)$ are both torsionfree, so we can tensor them by $\mathbf{Q}$ and get
\[p^*\colon H^{n-1}(V_{2n,n+1};\mathbf{Q})\to H^{n-1}(U(n);\mathbf{Q})\]
which factors through $H^{n-1}(SO(2n);\mathbf{Q})$.

\begin{figure}[htb]\label{spec-seq-so}
{\begin{center}\scriptsize
\begin{tikzpicture}
  \matrix (n) [matrix of math nodes,
    nodes in empty cells,nodes={minimum width=3ex,
    minimum height=2ex,outer sep=-2pt},
    column sep=2ex,row sep=1ex]{
H^0(V_{2n,n+1};H^{n-1}(SO(n-1);\mathbf{F})) & 0 & \cdots & 0 & H^{n-1}(V_{2n,n+1};H^{n-1}(SO(n-1);\mathbf{F}))\\
H^0(V_{2n,n+1};H^{n-2}(SO(n-1);\mathbf{F})) & 0 & \cdots & 0 & H^{n-1}(V_{2n,n+1};H^{n-2}(SO(n-1);\mathbf{F}))\\
\vdots & 0 & & 0 & \vdots\\
H^0(V_{2n,n+1};H^0(SO(n-1);\mathbf{F})) & 0 & \cdots & 0 & H^{n-1}(V_{2n,n+1};H^0(SO(n-1);\mathbf{F}))\\};
\draw[-stealth] (n-2-1.south east) -- (n-4-5.north west);
\end{tikzpicture}
\end{center}}
\caption{Part of the $E_{n-1}$-page of the spectral sequence for the fibration $SO(n-1)\to SO(2n)\to V_{2n,n+1}$. The coefficient field $\mathbf{F}$ is either $\mathbf{Q}$ if $n$ is odd or $\mathbf{Z}/2$ if $n$ is even.}
\end{figure}

By looking at the spectral sequence (see Figure \ref{spec-seq-so}) of the fibration
\[SO(n-1)\to SO(2n)\to V_{2n,n+1}\]
we see that the pullback of $g_n\in H^{n-1}(V_{2n,n+1};\QQ)$ to $H^{n-1}(SO(2n);\QQ)$ vanishes: it lives in the $E^{0,n-1}_{n-1}$ space of the spectral sequence and is necessarily killed by the differential coming from
\[E^{n-2,0}_{n-1}=H^{n-2}(SO(n-1);\QQ)\]
because
\[\OP{rank} H^{n-2}(SO(2n);\QQ)=\OP{rank} H^{n-2}(SO(n-1);\QQ)-1\]
since $H^{n-2}(SO(n-1);\QQ)$ additionally contains the suspension of the Euler class. As a consequence, $p^*F_i^*(g_n)=0$ when $n$ is odd and hence
\[\epsilon(f_1,f_2)=0.\]

When $n$ is even, we will show that $g_n$ pulls back to the suspension of $w_n$ in $H^{n-1}(SO(2n);\ZZ/2)$. To see this, consider the spectral sequence of the fibration $SO(n-1)\to SO(2n)\to V_{2n,n+1}$ for cohomology with coefficients in $\ZZ/2$. We get
\[H^{n-1}(SO(2n);\ZZ/2)=E^{n-1,0}_{\infty}\oplus E^{0,n-1}_{\infty}\]
and since $H^{n-1}(SO(2n);\ZZ/2)$ has the same rank as
\[E^{n-1,0}_{n-1}\oplus E^{0,n-1}_{n-1}=H^{n-1}(SO(n-1);\ZZ/2)\oplus H^{n-1}(V_{2n,n+1};\ZZ/2)\]
we know that the differential $E^{n-2,0}_{n-1}\to E^{0,n-1}_{n-1}$ vanishes and the pullback of $g_n$ to $H^{n-1}(SO(2n);\ZZ/2)$ survives. After pulling it back further to the fibre $H^{n-1}(SO(n-1);\ZZ/2)$ it vanishes. This implies that the pullback of $g_n$ to $SO(2n)$ must be the unique nonzero element in the kernel of the projection
\[H^{n-1}(SO(2n);\ZZ/2)\to H^{n-1}(SO(n-1);\ZZ/2),\]
which is precisely the suspension of $w_n$. Pulling back further to $U(n)$ gives the suspension of the $\tfrac{n}{2}$th Chern class reduced modulo 2; therefore
\[\epsilon(f_1,f_2) = \Delta(f_1,f_2) \mod 2.\]

Finally we will prove the statement about frame maps which agree on the $(n-2)$-skeleton. If the Lagrangian frame maps are homotopic when restricted to the $(n-2)$-skeleton of $L$, it follows that the same is true for the $SO(n-1)$-equivariant Lagrangian frame maps $\OP{Fr}(f_i)^{\uparrow}$. It follows that the value of
\[(p \circ F_1)^*(g_n)-(p \circ F_1)^*(g_n)\]
on an $(n-1)$-cell of $L$ can be obtained by evaluating $g_n$ on a spherical class, which moreover factorises as
\[S^{n-1} \to U(n) \to V_{2n,n+1}.\]
Lemma \ref{lma-htpy} shows that $\epsilon(f_1,f_2)$ vanishes when $n\geq 5$.
\end{proof}

\subsection{Proof of Corollary \ref{maincor-unknot}}\label{sect-cor-prf}

To prove Corollary \ref{maincor-unknot}, recall that the minimal Maslov number of a monotone torus is two (see the proof of Corollary \ref{cor-unknot}), hence the Maslov class is primitive. If we take a linear projection $\sigma$ to $S^1$ representing the Maslov class then its gradient $\nabla\sigma$ is a constant nonzero vector field on the torus. By Lemma \ref{lma-HaeHir}, the Haefliger-Hirsch field of such a monotone Lagrangian torus is $J\nabla\sigma$.

All constant nonzero vector fields on the torus are homotopic through constant nonzero vector fields. Therefore if $f_1,f_2$ are monotone Lagrangian embeddings of the torus where the Maslov classes are given by $\mu_i=\sigma_i^*\left[\tfrac{d\theta}{2\pi}\right]$, $i=1,2$, for two linear projections $\sigma_1,\sigma_2$, we have that $\nabla\sigma_1$ and $\nabla\sigma_2$ are homotopic through nonvanishing vector fields to a unit vector field $v$. Now by Proposition \ref{prp-diff-class} we see that the Haefliger-Hirsch obstruction to smooth isotopy vanishes when $n$ is odd and at least five.

%%%%%%%%%%%%%%%%%%%%%%%%%%%%%%%%%%%%%

\section{Generalisations}\label{sect-gen}

Recall that by Corollary \ref{cor-borrelli} our unknottedness result Theorem \ref{unknot} is already known for $S^1\times S^3$ and $S^1\times S^7$. It would be nice to recover this result using our techniques and to extend it to other products $S^1\times S^{n-1}$. Unfortunately we do not know how to do this in general. In this section we prove unknottedness in sufficiently high dimensions (starting with $S^1\times S^4$) with either a restriction on the minimal Maslov number or the condition of monotonicity.
\begin{thm}\label{unknot-tweak}
Let $n\geq 5$. Then the smooth knot type of a Lagrangian embedding $S^1\times S^{n-1}\to\CC^n$ is determined by the $SO(n)$-equivariant homotopy type of the Lagrangian frame map restricted to the complement $L_x$ of a point $x\in L$ in the two cases:
\begin{itemize}
\item $n$ is even and the minimal Maslov number is $n$,
\item $n$ is odd and the Lagrangian embedding is monotone.
\end{itemize}
\end{thm}
\begin{rmk}
Note that no monotonicity assumption is needed when the minimal Maslov number is $n>4$ because such Lagrangian embeddings are automatically monotone. This is because the first cohomology has rank one and because there exists a relative homology class of discs with strictly positive area and Maslov number between $3-n$ and $n+1$ {\cite[Theorem 2.1 and subsequent Remark 5]{BirCie}}. Since this disc is essential in relative homology its Maslov number is a nonzero multiple of $n$ by the assumption on minimal Maslov number. When $n>4$, the only nonzero multiple of $n$ between $3-n$ and $n+1$ is $n$. When $n$ is even, examples of Lagrangian embeddings satisfying the hypotheses of the theorem are well-known and can be constructed by a suitable Pol\-te\-ro\-vich surgery on the standard (Whitney) immersed exact Lagrangian sphere \cite{Pol}.

When $n$ is odd the odd-dimensional cohomology of the universal cover of $S^1\times S^{n-1}$ vanishes and hence (using monotonicity) Damian's theorem implies that the Lagrangian has minimal Maslov number 2. Moreover his theorem gives, as usual, a moduli space of holomorphic Maslov 2 discs with boundary on the Lagrangian such that the evaluation map $\mathcal{M}_{0,1}(L,\beta)\to L$ has degree one. Examples of monotone Lagrangians satisfying the hypotheses of the theorem when $n$ is odd can be constructed either by Polte\-ro\-vich surgery on the Whitney sphere as above, or by applying the construction of Audin-Lalonde-Polterovich \cite{ALP} to exact Maslov zero Lagrangian immersions of even dimensional spheres (which exist by the h-principle for exact Lagrangian immersions).
\end{rmk}
The strategy of proof is the same as for Theorem \ref{unknot}:
\subsection*{\texorpdfstring{$n=2k+2$}{n=2k+2} even}
As in Lemma \ref{lma-monopush} we use the projection $\sigma\colon S^1\times S^{n-1}\to S^1$ to find a nearby Lagrangian embedding $S^1\times S^{n-1}=L'\to\CC^n$ (the graph of $\epsilon d\sigma$) with minimal Maslov number $n$ and smaller monotonicity constant. We will pick a suitable almost complex structure $J$ and consider a suitable moduli space (see Lemma \ref{lma-moduli}) $M(\gamma,J)$ of $J$-discs with boundary on $L'$ representing the relative class $\beta\in H_2(\CC^n,L';\ZZ)\cong H_1(L';\ZZ)$ with Maslov number $2k+2$. We will see that
\begin{itemize}
\item this moduli space admits a degree one evaluation map to $L'$,
\item by choosing $J$ suitably we can assume that all the discs in this moduli space avoid $L$.
\end{itemize}
This will imply that $\nabla\sigma$ is the Haefliger-Hirsch field so that Proposition \ref{prp-diff-class} will then apply, proving Theorem \ref{unknot-tweak}. We begin by specifying the moduli space $M(\gamma,J)$.
\begin{lma}\label{lma-moduli}
Let $\gamma\colon S^1\to L'$ be a generic embedded loop representing the class $\beta\in H_1(L';\ZZ)$ and $J$ a generic almost complex structure. Let $\mM_{0,2}(\beta,J)$ denote the moduli space of $J$-discs with boundary on $L'$ representing the class $\beta$ and two boundary marked points. Write $\ev_1,\ev_2\colon\mM_{0,2}(\beta,J)\to L'$ for the evaluation maps at the two marked points. Let $M(\gamma,J)=\ev_2^{-1}(\gamma(S^1))$. Then the evaluation map
\[\ev_1|_{M(\gamma,J)}\colon M(\gamma,J)\to L'\]
has degree $\pm 1$.
\end{lma}
\begin{proof}
This is just a geometric interpretation of one of the $E_1$-differentials in the Biran-Cornea spectral sequence: the solid arrow in Figure \ref{spec-seq}.
\begin{figure}[htb]\label{spec-seq}
{\begin{center}
\begin{tikzpicture}
  \matrix (n) [matrix of math nodes,
    nodes in empty cells,nodes={minimum width=4ex,
    minimum height=2ex,outer sep=-2pt},
    column sep=2ex,row sep=1ex]{
0 & 0 & \ZZ t^{-1}\\
0 & 0 & \ZZ t^{-1}\\
0 & 0 & 0\\
0 & \ZZ & \ZZ t^{-1}\\
0 & \ZZ & \ZZ t^{-1}\\
0 & 0 & 0\\
\ZZ t & \ZZ & 0\\
\ZZ t & \ZZ & 0\\};
\draw[-stealth] (n-4-3.west) -- (n-4-2.east);
\draw[-stealth,dotted] (n-4-3.west) -- (n-3-1.south east);
\draw[-stealth,dotted] (n-5-3.west) -- (n-5-2.east);
\draw[-stealth,dotted] (n-5-3.west) -- (n-4-1.south east);
\end{tikzpicture}
\end{center}}
\caption{The $E_1$-page of the Biran-Cornea spectral sequence for a monotone Maslov 4 Lagrangian $S^1\times S^3$ in $\CC^2$ with $E_1$ (horizontal) and $E_2$ (knight's move) differentials indicated.}
\end{figure}
Since the spectral sequence must collapse at the $E_2$ stage this differential must be an isomorphism over $\ZZ$. The differential is multiplication by $t$ times the degree of the evaluation map we are interested in. Therefore this degree is $\pm 1$.
\end{proof}

The $n=2k+2$ case of Theorem \ref{unknot-tweak} will now follow from the following result.
\begin{prp}\label{prp-avoid-tweak}
There exists a loop $\gamma$ and a regular almost complex structure $J$ (obtained by neck-stretching) such that the discs in $M(\gamma,J)$ avoid $L$.
\end{prp}
To prepare for the proof of this result, we state a version Lazzarini's decomposition theorem for punctured holomorphic discs with boundary on a Lagrangian.

\begin{lma}[Compare with \cite{Laz,Laz2}]\label{lma-laz}
Let $X$ be a symplectic manifold with strong convex or concave contact-type boundary, let $\bar{X}$ denote its completion and $L\subset X$ be a compact Lagrangian submanifold. Let $J$ be a compatible almost complex structure on $\bar{X}$ adapted to the contact structure on the ends. Suppose that $u\colon\Sigma\to\bar{X}$ is a finite-energy punctured holomorphic disc with boundary on $L$. There exists a somewhere-injective finite-energy punctured disc $v\colon\Sigma'\to\bar{X}$ with boundary on $L$ such that $v(\Sigma')\subset u(\Sigma)$.
\end{lma}

We defer the proof of this lemma to Section \ref{app-laz}.

\begin{proof}[Proof of Proposition \ref{prp-avoid-tweak}]
Assume that Proposition \ref{prp-avoid-tweak} is false and that for all $\gamma$ and $J$ some disc in $M(\gamma,J)$ intersects $L$.

Fix the loop $\gamma$ in $L'$ and let $J_t$ be a neck-stretching sequence of almost complex structures for $L$. Choose a Weinstein neighbourhood of $L$ and let $V$ denote its complement. We will later choose $I=J_t|_V$ appropriately to derive a contradiction. By Lemma \ref{lma-moduli}, for all $t$ there is a $J_t$-holomorphic disc in $M(\gamma,J_t)$, that is a disc with boundary on $L'$ representing the minimal area class $\beta$ and (by assumption) passing through $L$.

We can extract a Gromov-Hofer convergent subsequence of these discs whose limit is a holomorphic building $u=u_{\overline{W}}\cup u_{\overline{S}_1}\cup\cdots\cup u_{\overline{S}_{\ell}}\cup u_{\overline{V}}$ where $u_{\overline{W}}$ denotes the component in $\overline{W}=T^*L$ (the symplectic completion of the Weinstein neighbourhood), $u_{\overline{V}}$ denotes the component in $\overline{V}=\CC^n\setminus L$ and $u_{\overline{S}_k}$ denote components in the intermediate symplectisation levels. Let $v$ be the component of $u_{\overline{V}}$ with boundary on $L'$. Note that the boundary of $v$ passes through the loop $\gamma$ by construction.
\begin{clm}\label{clm-only-cpnt}
There are no other components in $u_{\overline{V}}$.
\end{clm}
\begin{proof}
Another component would compactify to give a topological surface with boundary on $L$, and would represent a relative homology class in $H_2(\CC^n,L;\ZZ)$ with nonzero area since it is holomorphic. Since the monotonicity constant of $L$ is strictly bigger than that of $L'$ the area of this component is strictly bigger than the area of $\beta$, which contradicts Lemma \ref{lma-area-conserved}.
\end{proof}
\begin{clm}
The curve $v$ has only contractible geodesics as asymptotes.
\end{clm}
\begin{proof}
Assume that $v$ has a noncontractible geodesic as asymptote. Since the building as a whole has genus zero this asymptote must be the positive asymptote of another genus zero building in $\overline{W}\cup \overline{S}_1\cup\cdots\cup \overline{S}_{\ell'}\cup \overline{V}$ which glues topologically to give a plane in $\CC^n$. Since the geodesic is noncontractible this building cannot lie entirely in $\overline{W}\cup \overline{S}_1\cup\cdots\cup \overline{S}_{\ell'}$ or else it would glue topologically to give a nullhomotopy of the geodesic. Therefore there must be another component of $\overline{V}$, contradicting Claim \ref{clm-only-cpnt}.
\end{proof}
\begin{clm}\label{clm-exist-somewhere-inj}
The curve $v$ is somewhere injective.
\end{clm}
\begin{proof}
By Lemma \ref{lma-laz} we can extract a somewhere injective punctured disc $w'$ whose image is a subset of $w$. Suppose that $w'$ is not just a reparametrisation of $w$. Since the asymptotes are contractible geodesics in $L$ we can cap them off with discs in $L$ with no symplectic area and obtain a topological disc in $\CC^n$ with boundary on $L$ whose area is strictly smaller than that of $\beta$, which contradicts minimality of $\beta$.
\end{proof}
Note that Claim \ref{clm-exist-somewhere-inj} applies to all finite-energy punctured discs in the same moduli space as $v$: we only use the fact that the asymptotes are contractible geodesics and that the result of topologically capping these asymptotes is homologous to $\beta$.

In the standard way \cite{Dra}, we can achieve transversality for moduli spaces of somewhere-injective finite-energy punctured discs by perturbing $I=J_t|_V$. In particular the moduli space $S$ of punctured discs containing $v$ is smooth and of the expected dimension. Similarly, if we equip punctured discs with a boundary marked point, we can assume that the resulting evaluation map $\ev\colon S\to L'$ is transverse to $\gamma$.
\begin{clm}\label{lma-exp-dim}
When $n>4$ the expected dimension of $\ev^{-1}(\gamma)\subset S$ is negative. In particular, when $I$ is chosen generically this set is empty.
\end{clm}
\begin{proof}
Using Equation \eqref{eq-index-form} from Section \ref{sect-index-form}, the expected dimension formula for punctured discs in $V$ with $s_-$ negative punctures asymptotic to Reeb orbits, the $i$th of which covers a contractible geodesic in $S^{n-1}$ with multiplicity $m_i$, is
\[(n-3)(1-s_-)+\mu-\sum_{i=1}^{s_-}(2m_i-1)(n-2)\]
Since $\mu=n$, the expected dimension is:
\[2n-3-\sum_{i=1}^{s_-}\left(n-3+(2m_i-1)(n-2)\right)\]
If we add a marked point on the boundary and require this to pass through the codimension $n-1$ loop $\gamma$ then the expected dimension becomes
\[2n-3-\sum_{i=1}^{s_-}\left(n-3+(2m_i-1)(n-2)\right)+1-(n-1)\leq 4-n\]
with equality if and only if $s_-=1$ and $m_1=1$. When $n>4$ this implies the claim.
\end{proof}
Since $v$ is supposed to belong to this empty moduli space we get a contradiction. This completes the proof of Proposition \ref{prp-avoid-tweak} and therefore the proof of Theorem \ref{unknot-tweak} in the case $n=2k+2$.
\end{proof}

\subsection*{\texorpdfstring{$n=2k+1$}{n=2k+1} odd}

In this case Damian's theorem from \cite{Dam} implies that, for a regular $J$, the evaluation map
\[\ev\colon\mathcal{M}_{0,1}(\beta,J)\to L\]
(from the moduli space of $J$-discs with boundary on $L$ representing the class $\beta=[S^1]\times\{\star\}\in H_1(L;\ZZ)$ and having one boundary marked point) has nonzero degree (in fact degree $\pm 1$). We argue in the usual way to prove that the vector field $\partial_{\theta}$ ($\theta$ being the coordinate on $S^1$) is a Haefliger-Hirsch field: push $L$ off along $J\partial_{\theta}$ to obtain a Lagrangian $L'$ with smaller monotonicity constant and study the moduli space of discs with boundary on $L'$ (for which the corresponding evaluation map still has nonzero degree). Assuming that for every $J$ there is a $J$-disc on $L'$ in the class $\beta$ which intersects $L$ we use a neck-stretching sequence $J_t$ (stretching around $L$) and extract a Gromov-Hausdorff convergent subsequence of $J_t$-discs which intersect $L$. One component is a plane and it is possible that this plane lives in the Weinstein neighbourhood of $L$. As in the case $n=2k+2$ we will argue that this cannot occur for generic neck-stretching sequences provided that $n>4$. Claims 1-3 still apply and we just need to understand what replaces Claim 4.

By Claims 1-3, the part $u_{\overline{V}}$ of the limit building living in the complement of $L$ consists of a single punctured disc with boundary on $L'$ and having contractible geodesics as asymptotes. The expected dimension for moduli spaces of such discs is
\[(n-3)(1-s_-)+\mu-\sum_{i=1}^{s_-}(2m_i-1)(n-2)\]
where now $\mu=2$ and $m_i$ is the number of times the $i$th Reeb orbit wraps around the underlying simple geodesic. The expected dimension is therefore
\[n-1-\sum_{i=1}^{s_-}(n-3+(2m_i-1)(n-2))\]
which is at most $n-1-(n-3+n-2)=4-n$ in the worst case $s_-=1$, $m_1=1$. If $n>4$ this is negative so this kind of breaking is generically prohibited.

This completes the proof of Theorem \ref{unknot-tweak} in the case $n=2k+1$. \qedhere

\subsection{Proof of Corollary \ref{maincor-sphere-unknot}}

Let $f_1,f_2\colon L=S^1 \times S^{n-1} \to \CC^n$ be two Lagrangian embeddings which both satisfy one of the assumptions in Theorem \ref{unknot-tweak}. Let $\sigma$ be the projection onto the $S^1$-factor. Since $n>2$, the standard cell decomposition of $L$ has $(n-2)$-skeleton $S^1 \times \{ pt \}$. If $n$ is even then the Maslov class is $\pm n[d\sigma]$ by assumption. If $n$ is odd then by monotonicity and Damian's theorem \cite{Dam} we know that the Maslov class is $\pm 2[d\sigma]$. In either case, by reparametrising with a reflection of $S^1$ if necessary, we can assume that the Maslov class of $f_1$ and the Maslov class of $f_2$ agree, which ensures that the Lagrangian frame maps are homotopic when restricted to the $(n-2)$-skeleton of $L$.

By Theorem \ref{unknot-tweak}, Proposition \ref{prp-diff-class} applies, in particular the last statement showing that the difference class $\epsilon(f_1,f_2)$ vanishes. This finishes the proof of Corollary \ref{maincor-sphere-unknot}.

\section{Constructing knotted Lagrangian tori}\label{sect-knots}

When $n$ is even we will construct smoothly non-isotopic monotone
Lagrangian tori\footnote{The construction given here has some
errors. See the erratum at the end.}.

Consider the Clifford (product) torus. If we use the trivialisation of $TT^n$ coming from its structure as a Lie group then we get a trivialisation of $T_{\CC}T^n$ and with respect to this trivialisation the Lagrangian frame map $T^n\to U(n)$ of the Clifford torus is just just $i$ times the inclusion of a maximal torus. For the inclusion of a maximal torus, the suspension of the Chern class $c_{n/2}$ is nontrivial and not divisible by two.

Using the h-principle for exact Lagrangian immersions, let $g$ be an exact Lagrangian immersion whose Lagrangian frame map is nullhomotopic. Apply the Audin-Lalonde-Polterovich construction to $g$ to obtain an embedded Lagrangian diffeomorphic to $S^1\times T^{n-1}$ in $\CC^n$ which is monotone by exactness of the immersion and by the way we have chosen the Gauss map. With respect to the same trivialisation of $TT^n$ the suspension of $c_{n/2}$ vanishes.

By construction the difference class between these two monotone Lag\-ran\-gian embeddings is nonzero so they are not smoothly isotopic. Analogous examples of smoothly knotted $S^1\times S^3$s with non-homotopic Gauss maps were constructed by Borrelli \cite{Bor}.

\section{Appendix I: Index formula for pseudoholomorphic curves}\label{sect-index-form}

Consider the symplectic manifold $V=\CC^n\setminus L$ with a negative cylindrical end. In this section we will describe the Fredholm index of the linearised $\overline{\partial}$-operator for pseudoholomorphic discs in $V$ having boundary on a Lagrangian submanifold $L'\subset V$ and internal boundary punctures asymptotic to Reeb orbits of the negative cylindrical end. Since we are interested in the case when the negative end corresponds to $(-\infty,0]\times S_{\rho,g}^*L$, where the metric $g$ is non-degenerate in the Bott sense, we also consider index formulas for pseudoholomorphic curves having punctures asymptotic to Reeb orbits which are non-degenerate in the Bott sense.
       
\subsection{The generalised Conley-Zehnder index}
We start with a brief description of the Conley-Zehnder indices of the Reeb orbits in this situation. The Conley-Zehnder index $\mu_{CZ}(\gamma)$ of a Reeb orbit $\gamma$ inside a contact manifold $(Y,\xi=\ker \lambda)$ can be computed as follows, following {\cite[Remark 5.4]{RobSal}}. First, we fix a symplectic trivialisation of the contact distribution $\xi$ along $\gamma$, in which the linearised Reeb flow is expressed as a path of symplectic matrices $\Psi_t$. One gets an induced path of Lagrangian planes inside $(\CC^n \oplus \CC^n,(-\omega_0) \oplus \omega_0)$  parametrised by
\[(\mathrm{Id},\Psi_t) \colon \CC^n \to \CC^n \oplus \CC^n\]
for which one can compute the Maslov index as defined in {\cite[Section 2]{RobSal}} with respect to the Lagrangian reference plane consisting of the diagonal. This is the Conley-Zehnder index.

Observe that this index is defined even in the case when $1$ is an eigenvalue of the return-map of the linearisation, but that it may take half-integer values in this case. When there is a Bott manifold $S$ of Reeb orbits, we will use $\mu_{CZ}(S)$ to denote the Conley-Zehnder index of a Reeb orbit in this family.

\subsection{The Fredholm index for a pseudoholomorphic curve with punctures in the Bott case}
In {\cite[Proposition 2.7]{Bou}} the formula for the Fredholm index of the linearised $\overline{\partial}$-operator for a closed pseudoholomorphic curve inside a symplectic $2n$-dimensional manifold $\overline{X}$ with cylindrical ends is generalised to the case where the Reeb orbits are nondegenerate in the Bott sense. It is shown that a closed pseudoholomorphic curve $C$ of genus $g$ having internal punctures asymptotic to Reeb orbits in the families $S^+_1, \hdots, S^+_{s^+}$ and $S^-_1,\hdots,S^-_{s^-}$ at positive and negative ends, respectively, the Fredholm index satisfies
\begin{align*}
\mathrm{index}(C)&=(n-3)(2-2g-s^+-s^-)+2c_1^{\mathrm{rel}}(C)+ \\
&\ \ \ \ \ + \sum_{i=1}^{s^+}\left(\mu_{CZ}(S^+_i)+\frac{1}{2}\dim S^+_i\right) -\sum_{i=1}^{s^-}\left(\mu_{CZ}(S^-_i)-\frac{1}{2}\dim S^-_i\right).
\end{align*}
Here $c_1^{\mathrm{rel}}(C)$ denotes the first Chern number of the bundle $T\overline{X}$ pulled back to $C$ and extended over the punctures using the trivialisation of $T\overline{X}|_{\gamma} =\CC \oplus \xi|_{\gamma}$ chosen above along the Reeb orbits.

In the case when the curve $C$ has boundary $\partial C$ on a Lagrangian submanifold $L' \subset \overline{X}$, given any trivialisation of $T\overline{X}$ along the boundary, there is an induced Maslov index of $L'$ along this boundary which we denote by $\mu(\partial C)$. One can deduce that
\begin{align}
\nonumber\mathrm{index}(C)&= (n-3)(1-2g-s^+-s^-)+\mu(\partial C)+2c_1^{\mathrm{rel}}(C) +\\
\label{eq-gen-index}&\ \ \ \ \  + \sum_{i=1}^{s^+}\left(\mu_{CZ}(S^+_i)+\frac{1}{2}\dim S^+_i\right) -\sum_{i=1}^{s^-}\left(\mu_{CZ}(S^-_i)-\frac{1}{2}\dim S^-_i\right).
\end{align}
Here $c_1^{\mathrm{rel}}(C)$ denotes the first Chern number of $T\overline{X}$ pulled back to $C/\partial D$ using the trivialisation of $T\overline{X}|_{\partial  D}$ chosen above and extended to the punctures as before.

In the case when the first Chern class $c_1$ vanishes for $\overline{X}$, the above Conley-Zehnder index is canonically defined for nullhomologous Reeb orbits in any trivialisation induced by a choice of bounding chain. Likewise, the Maslov index is canonically defined for any path on $L'$ which is nullhomologous in $\overline{X}$. This follows from the fact that two different choices of bounding chains $A$ and $B$ will give rise to a difference of $2c_1(A-B)$ in the respective index. Finally, in the case $c_1=0$ and when the trivialisations used are induced as above, it also follows that the term $2c_1^{\mathrm{rel}}(C)$ vanishes for these choices of trivialisations.

\subsection{The index formula for discs in \texorpdfstring{$\overline{V}$}{V} with internal punctures }
Consider the symplectic manifold $\overline{V}=\CC^n \setminus L$, where $L \cong S^1 \times S^{n-1}$ is a Lagrangian submanifold. We view $\overline{V}$ as a symplectic manifold with a negative end corresponding to $(-\infty,0] \times S^*_{\rho,g}(S^1 \times S^{n-1})$, where $g$ is the product metric on $S^1 \times S^{n-1}$ for the round metric on $S^{n-1}$. We are now ready to show the following result for discs inside $\overline{V}$ having boundary on a Lagrangian submanifold $L' \subset \overline{V}$. 

For the Maslov index of a closed curve on $L' \subset \overline{V}$ we use the trivialisation of $T\overline{V} \subset T\CC^n$ induced by the canonical trivialisation of $T\CC^n$. Observe that any loop on $L'$ is contractible inside $\overline{V}$ whenever $n>2$ and that the induced trivialisation by any chain bounding a loop on $L'$ agrees with this trivialisation.

\begin{lma}
A pseudoholomorphic disc $D$ with boundary on $L' \subset \overline{V}$ having a number $s^-$ of internal punctures asymptotic to the families $S^-_1,\hdots,S^-_{s^-}$ of Reeb orbits at the negative end of $V$, where $S_i$ corresponds to a family of $m_i$-multiple covers of closed contractible geodesics on $S^1 \times S^{n-1}$, has Fredholm index
\begin{equation}\label{eq-index-form}\mathrm{index}(D)= (n-3)(1-s^-)+\mu(\partial C)-\sum_{i=1}^{s^-}(2m_i-1)(n-2).\end{equation}
\end{lma}

\begin{proof}
For the trivialisation of the contact distribution $\xi|_{\gamma}$ along a Reeb orbit $\gamma$ of $S_{\rho,g}^*L$ we choose the trivialisation which is induced by a trivialisation of the vertical subbundle of $\xi$, which is a Lagrangian subspace. Observe that this trivialisation agrees with the canonical trivialisation of $\xi$ along any Reeb orbit on $S_{\rho,g}^*L$ which is nullhomologous. Moreover, since $c_1$ vanishes for $\overline{V}$, this means that the term $c_1^{\mathrm{rel}}$ vanishes in the index formula.

Consider the energy functional
\[E(\eta)=\int_{S^1} \|\eta(\theta)\|^2d\theta\]
for closed curves $\eta$ in $L$. Let $H$ denote the Hessian of $E$ at a critical point $\gamma$, which corresponds to a closed geodesic on $L$ parametrised by path of constant speed. We let $\iota(\gamma)$ and $\nu(\gamma)$ denote the dimension of negative-definite eigenspace and nullity of $H$ at $\gamma$, respectively.

We let $\gamma$ be a geodesic on $S^1 \times S^{n-1}$ and $\widetilde{\gamma}$ be the Reeb orbit in $S^*_{\rho,g}(S^1 \times S^{n-1})$ corresponding to the cogeodesic lift. The lemma follows by combining Lemma \ref{lma-sphere-index} below together with Equation \eqref{eq-gen-index} above, and the formula
\[ \mu_{CZ}(\widetilde{\gamma})=\iota(\gamma)+\frac{1}{2}\nu(\gamma)\]
proved in {\cite[Equation 60]{CieFra}}, which holds in the canonical trivialisation.
\end{proof}

\begin{lma}\label{lma-sphere-index}
Let $S^1 \times S^{n-1}$ be endowed with the product metric, where the metric on the factor $S^{n-1}$ is the round metric. A closed contractible geodesic $\gamma$ on $S^1 \times S^{n-1}$, which moreover is the $m$-fold cover of a simply covered geodesic, has Morse-index and nullity satisfying
\[ \iota(\gamma)=(2m-1)(n-2),\qquad \nu(\gamma)=n.\]
\end{lma}
\begin{proof}
Since geodesics on $S^1\times S^{n-1}$ project to geodesics on either factor, it follows that any geodesic which starts and ends on the hypersurface $\{t\} \times S^{n-1}$ must either be a geodesic contained entirely in this hypersurface, or must wrap around the $S^1$-direction a nonzero number of times. Furthermore, a geodesic contained in this hypersurface is a geodesic on the round $S^{n-1}$.

From this it follows that a contractible geodesic on $S^1 \times S^{n-1}$ is contained in such a hypersurface, and that broken Jacobi fields along such a geodesic correspond bijectively to broken Jacobi fields on the corresponding geodesic on the round $S^n$. In addition, the only non-broken Jacobi field along a closed geodesic in $\{t\} \times S^{n-1}$ which does not arise as a Jacobi field on $S^{n-1}$ is the Jacobi field induced by a rotation of the $S^1$-factor.

The formulae on {\cite[Page 346]{Bott}} for the index and nullity of a geodesic on $S^{n-1}$ now implies the result.
\end{proof}

\section{Appendix II: Lazzarini's theorem for punctured discs}\label{app-laz}

We will prove Lemma \ref{lma-laz}:

\begin{lma}
Let $X$ be a symplectic manifold with contact-type boundary, let $\overline{X}$ denote its completion and let $L\subset X$ be a compact Lagrangian submanifold. Let $J$ be a compatible almost complex structure on $\overline{X}$ adapted to the contact structure on the ends. Suppose that $u\colon\Sigma\to\overline{X}$ is a finite-energy punctured $J$-holomorphic disc with boundary on $L$. There exists a somewhere-injective finite-energy punctured $J$-holomorphic disc $v\colon\Sigma'\to\overline{X}$ with boundary on $L$ such that $v(\Sigma')\subset u(\Sigma)$.
\end{lma}
This was proved for non-punctured discs in \cite{Laz2}; we will explain how to modify that proof to incorporate punctures. Recall that the punctures of punctured finite-energy $J$-holomorphic curves are asymptotic to cylinders on Reeb orbits. Since $L$ is disjoint from the cylindrical end of $\overline{X}$ we can analyse the boundary and the punctures of the disc separately. There are a finite number punctures and hence a finite number of asymptotically cylindrical ends of $u$. It follows from work of Siefring that there is a compact subset $K\subset\overline{X}$, containing $X$, with the following properties.
\begin{enumerate}
\item[(a)] In the complement of $K$, each asymptotically cylindrical end of $u$ is an unbranched multiple cover of an embedded $J$-holomorphic half-cylinder. This is {\cite[Corollary 2.6]{Sief}}: the absence of branch points for sufficiently large $K$ follows from the explicit asymptotic form of the covering map.
\item[(b)] In the complement of $K$, the half-cylinders corresponding to different punctures are either disjoint or have the same image. This is {\cite[Corollary 2.5]{Sief}}.
\end{enumerate}
Let $\Sigma_0=u^{-1}(\overline{X}\setminus K)\subset\Sigma$ be the corresponding neighbourhood of the punctures.

\begin{dfn}
Given a punctured finite-energy $J$-holomorphic disc
\[u\colon\Sigma\to\overline{X}\]
with boundary on $L$, let $C(u)=\{x\in\Sigma\ :\ d_xu=0\}$ denote the set of critical points of $u$.
\end{dfn}

\begin{lma}
If $u$ is nonconstant then $C(u)$ is finite and, for any $x\in\overline{X}$ the preimage $u^{-1}(x)$ is finite.
\end{lma}
\begin{proof}
By {\cite[Theorem 3.5]{Laz2}} we know that the set of critical points of $\Sigma\setminus\Sigma_0$ is discrete, hence finite because $\Sigma\setminus\Sigma_0$ is compact. There are no critical points of $u$ in $\Sigma_0$ by (a). Thus $u$ has finitely many critical points.

Similarly each $x\in\overline{X}$ has only finitely many preimages under $u$: for $x\in\overline{X}\setminus K$ this is clear from (a) and (b); for $x\in K$ it follows from {\cite[Theorem 3.5]{Laz2}}.
\end{proof}
Now Lazzarini's strategy to extract a somewhere-injective disc is to consider the {\em frame} of $u$:
\begin{dfn}
Define a relation $\mathcal{R}$ on $\Sigma\setminus C(u)$ by defining $z\mathcal{R}z'$ if and only if, for any neighbourhoods $V,V'$ of $z,z'$ there are open neighbourhoods $z\in W\subset V$ and $z'\subset W'\subset V'$ such that $u(W)=u(W')$. Define $\overline{\mathcal{R}}$ to be the closure of $\mathcal{R}\subset\Sigma\times\Sigma$. 
\end{dfn}
Using (a) and (b) we can completely describe the relation $\overline{\mathcal{R}}$ on $\Sigma_0$: each component $U_i$ of $\Sigma_0$ is a punctured open disc and $U_i/\overline{\mathcal{R}}$ is also a punctured open disc, the map $U_i\to U_i/\overline{\mathcal{R}}$ being an unbranched cover. It is also possible that some components $U_i$ and $U_j$ are identified by the relation.

Lazzarini's analysis of the relation applies unchanged to $\Sigma\setminus\Sigma_0$. In particular {\cite[Theorem 4.7]{Laz2}} the set of points $\overline{\mathcal{R}}(\partial\Sigma)$ related to points on the boundary forms an embedded graph whose vertices lie on points of $u^{-1}(u(C(u)))$. This is called the frame $\mathcal{W}(u)$ and is clearly disjoint from $\Sigma_0$.

Inside the frame is the subset $\mathcal{W}_1(u)$ consisting of points contained on cycles in the graph, that is continuous injections $S^1\to\mathcal{W}(u)$. If we compactify $\Sigma$ to a disc $\Sigma'$ by filling in the punctures then Lazzarini's argument {\cite[Section 5]{Laz2}} guarantees a simply-connected component $D$ of $\Sigma'\setminus\mathcal{W}_1(u)$ which is topologically a disc such that the restriction of $u$ to $D\cap\Sigma$ is a multiple cover of a somewhere-injective punctured $J$-holomorphic disc. The only part of this argument which needs modification to compensate for the punctures is {\cite[Proposition 5.9]{Laz2}}, putting the structure of a Riemann surface on $\Sigma/\overline{\mathcal{R}}$. It is clear from our description of the relation $\overline{\mathcal{R}}$ on $\Sigma_0$ that $\Sigma_0/\overline{\mathcal{R}}$ is a union of punctured open discs.

\section*{Acknowledgements}

We would like to acknowledge helpful discussions and communications with Paul Biran, Felix Schlenk and Chris Wendl. The paper \cite{Bor} proved extremely helpful in crystallising our proof. We would also like to thank the anonymous referee for their careful critique of the paper which has helped us to clarify the exposition at several key points.

\bibliographystyle{plain}
\bibliography{unlink-bib}
\end{document}